\documentclass[11pt,reqno,a4paper]{amsart}
\usepackage[T1]{fontenc}
\usepackage[utf8]{inputenc}
\usepackage{lmodern,amssymb,dsfont,enumerate}
\usepackage[pagebackref]{hyperref}
\usepackage[portrait,a4paper,margin=3cm]{geometry}

\usepackage{tikz}
\usetikzlibrary{arrows}

\title[Random Markov matrices with heavy tailed weights]%
{Spectrum of large random Markov chains: Heavy-tailed weights on the oriented complete graph}

\date{To appear in Random Matrices: Theory and Applications (2017)}

\author{Ch.~Bordenave}
\address[Charles~Bordenave]{IMT UMR5219 CNRS Universit\'e de Toulouse, France}
\email{charles.bordenave(at)math.univ-toulouse.fr}
\urladdr{http://www.math.univ-toulouse.fr/~bordenave/}

\author{P.~Caputo}
\address[Pietro~Caputo]{Dipartimento di Matematica, Universit\`a Roma Tre, Italy}
\email{caputo(at)mat.uniroma3.it}
\urladdr{http://www.mat.uniroma3.it/users/caputo/}

\author{D.~Chafa\"\i}
\address[Djalil~Chafa{\"{\i}}]{CEREMADE UMR7534 CNRS Université Paris-Dauphine PSL IUF, France}
\email{djalil(at)chafai.net} 
\urladdr{http://djalil.chafai.net/}

\author{D.~Piras}
\address[Daniele~Piras]{Dipartimento di Matematica, Universit\`a Roma Tre, Italy}
\email{piras(at)mat.uniroma3.it}


\newtheorem{theorem}{Theorem}[section]%
\newtheorem{lemma}[theorem]{Lemma}
\newtheorem{proposition}[theorem]{Proposition}%
\newtheorem{example}[theorem]{Example}%
\newtheorem{remark}[theorem]{Remark}%

\newcommand{\dC}{\mathbb{C}}\newcommand{\dD}{\mathbb{D}}
\newcommand{\dE}{\mathbb{E}}

\newcommand{\dN}{\mathbb{N}}
\newcommand{\dP}{\mathbb{P}}
\newcommand{\dR}{\mathbb{R}}

\newcommand{\dZ}{\mathbb{Z}}
  
 \newcommand{\bbD}{{\mathbb D}} 
\newcommand{\bbE}{{\mathds E}}

 \newcommand{\bbN}{{\mathbb N}} 
 \newcommand{\bbP}{{\mathds P}} 
 \newcommand{\bbR}{{\mathds R}}

\newcommand{\cA}{\mathcal{A}}
\newcommand{\cD}{\mathcal{D}}

\newcommand{\cL}{\mathcal{L}}

\newcommand{\cT}{\mathcal{T}}
\newcommand{\cU}{\mathcal{U}}

\newcommand{\al}{\alpha} 
\newcommand{\be}{\beta}
\newcommand{\veps}{\varepsilon}
\newcommand{\de}{\delta}
\newcommand{\la}{\lambda}
 
\newcommand{\g}{\gamma} 
    
\newcommand{\om}{\omega}

\newcommand{\vte}{\vartheta}
\newcommand{\G}{\Gamma}  
\newcommand{\te}{\theta}

\newcommand{\si}{\sigma}
\newcommand{\e}{\mathrm{e}}
\newcommand{\dd}{\mathrm{d}}

\renewcommand{\leq}{\leqslant}             
\renewcommand{\geq}{\geqslant}             
\newcommand{\eset}{\varnothing}          
\newcommand{\scalar}[2]{\langle #1 , #2\rangle}

\newcommand{\wt}{\widetilde}
\newcommand{\ind}{\mathbf{1}}

\newcommand{\tr}{\mathrm{tr}}

\newcommand{\bv}{\mathbf{v}}
\newcommand{\bu}{\mathbf{u}}

\newcommand{\supp}{\mathrm{supp}}

\newcommand{\dist}{\mathrm{dist}}

\renewcommand{\Im}{\mathfrak{Im}}

\newcommand{\bx}{{\mathbf{x}}}
\newcommand{\bk}{{\mathbf{k}}}
\newcommand{\wh}{\widehat}
\newcommand{\weak}{\rightsquigarrow}

\numberwithin{equation}{section}

\begin{document}

\begin{abstract}
  We consider the random Markov matrix obtained by assigning i.i.d.\
  non-negative weights to each edge of the complete oriented graph. In this
  study, the weights have unbounded first moment and belong to the domain of
  attraction of an alpha-stable law. We prove that as the dimension tends to
  infinity, the empirical measure of the singular values tends to a
  probability measure which depends only on alpha, characterized as the
  expected value of the spectral measure at the root of a weighted random
  tree. The latter is a generalized two-stage version of the Poisson weighted
  infinite tree (PWIT) introduced by David Aldous. Under an additional
  smoothness assumption, we show that the empirical measure of the eigenvalues
  tends to a non-degenerate isotropic probability measure depending only on
  alpha and supported on the unit disc of the complex plane. We conjecture
  that the limiting support is actually formed by a strictly smaller disc.
\end{abstract}
\thanks{Support: A*MIDEX project ANR-11-IDEX-0001-02 funded by the
  ``Investissements d'Avenir'' French Government program, managed by the
  French National Research Agency (ANR)}
\maketitle
\thispagestyle{empty}
{\footnotesize\tableofcontents}
\section{Introduction}

A natural way to construct a random $n\times n$ Markov matrix is to assign
i.i.d.\ (independent and identically distributed) non-negative weights
$X_{i,j}$ with a given law $\cL$ to each ordered pair $(i,j)$,
$i,j=1,\dots,n$, and then consider the stochastic matrix $M$ obtained by
normalizing each row with the corresponding row sum:
\begin{equation}\label{eq:mn}
  M_{i,j}:=\frac{X_{i,j}}{\rho_i} \quad\text{where}\quad\rho_i:=\sum_{j=1}^nX_{i,j}\,.%
\end{equation} 
If $\rho_i=0$ for some $i$, it is understood that $M_{i,i}=1$ and $M_{i,j}=0$
for all $j\neq i$. The random matrix $M$ has independent rows but
non-independent columns. The eigenvalues of $M$, that is the complex roots of
$\det(M-z)$, are denoted by $\la_1,\ldots,\la_n$. Since $M$ is a
stochastic matrix, one has $\la_1=1$ and $|\la_j|\leq 1$, $j=2,\dots,n$; see e.g.\ \cite[Chap.\ 8]{MR1084815}. If the
symmetry of the weights $X_{i,j}=X_{j,i}$ is imposed, then the resulting
Markov chain is reversible with respect to the row sum measure $\rho_i$, and has
real spectrum. Such reversible models have been studied in
\cite{BCCalea,BCChr}. Here we consider the case where $X_{i,j}$ and $X_{j,i}$
are independent. In this case the Markov chain is non-reversible and has
complex eigenvalues.

If $\cL$ has finite second moment, it was shown in \cite{BCCm} that the
spectrum is asymptotically uniformly distributed in the disc of radius
$\si/\sqrt n$, where $\si^2$ denotes the variance of $\cL$, i.e.\
\begin{equation}\label{eq:con}
  \frac 1 n \sum_{k=1} ^n \delta_{\la_k \sqrt{\tfrac{n}{\si^2}}}
  \underset{n\to\infty}{\weak}\cU
  \,,%
\end{equation} 
where $\cU$ is the uniform law on the unit disc $\bbD=\{z\in\dC:|z|\leq 1\}$
and $\weak$ denotes the weak convergence of probability measures with respect
to continuous and bounded test functions. Similar results were recently
obtained for discrete matrices with given row sums \cite{NguyenVu}, for the
ensemble of uniformly random doubly stochastic matrices
\cite{ChatterjeeDiaconisSly,Nguyen} and for random matrices with exchangeable entries \cite{ACW}.

In this paper we consider the infinite variance case.
We shall actually restrict our attention to the particularly interesting case
where the law $\cL$ of the entries has infinite first moment.
Our main results can be formulated as follows. We assume that for some
$\al\in(0,1)$, the random variables $X_{i,j}$ are i.i.d.\ copies of a random
variable $\bx$ satisfying the assumptions below:
\begin{enumerate}
\item[(H1)] 
  $\bx$ is a non-negative random variable such that 
  \begin{equation}\label{eq:H}
    c:=\lim_{t\to\infty} t^\al\, \dP ( \bx \geq t )>0 .
  \end{equation}
\item[(H2)] $\bx$ has a bounded probability density function. 
\end{enumerate}
It is well known that a random variable satisfying (H1) is in the domain of
attraction of an $\al$-stable law. An example of random variable satisfying
both (H1) and (H2) is $\bx= U^{-1/\al}$, where
$U$ 
is any bounded non-negative random variable with continuous probability density
$\varphi$ on $[0,\infty)$ such that $\varphi(0)>0$.

We recall that for every fixed $i$, the ordered rearrangement of the random
row vector $\{M_{i,j},\,j=1,\dots, n\}$ converges weakly to the
Poisson-Dirichlet distribution with parameter $\al\in(0,1)$; see
\cite{PitmanYor}, \cite[Lem.~2.4]{BCChr}. Thus, one expects that the
distribution of eigenvalues of $M$ converges to a nontrivial probability
measure on the unit disc $\dD$ without any further rescaling. This is what we
prove in Theorem \ref{th:girko} below. The proof, following the ``objective
method'' philosophy \cite{aldoussteele}, will be based on the construction of
an infinite random tree that can be identified as a suitable local limit of
the random matrix $M$.

As usual for the non Hermitian setting, the analysis of the eigenvalues starts
with the understanding of the asymptotic behavior of the singular values of
$M-z$, for $z\in\dC$. Here and below, if no confusion arises, we write $z$ for
the diagonal matrix $zI$. Consider the singular values $s_{k,z}$,
$k=1,\dots,n$, that is the eigenvalues of $\sqrt{(M-z)(M-z)^*}$, and write
$\nu_{M,z}$ for the associated empirical distribution:
\begin{equation}\label{eq:num}
  \nu_{M,z} := \frac 1 n \sum_{k=1} ^n \delta_{s_{k,z}}.
\end{equation}

\begin{theorem}[Singular values]\label{th:numz}
  If $(H1)$ holds, then for each $z\in\dC$ there exists a probability measure
  $\nu_{\al,z}$ on $[0,\infty)$ depending only on $\al$ and $|z|$ such that
  almost surely
  \begin{equation}\label{eq:numz1}
    \nu_{M,z} \underset{n\to\infty}{\weak} \nu_{\al,z}.
  \end{equation}
  For any $z\in\dC$ the measure $\nu_{\al,z}$ has unbounded support and satisfies, for all $\la>0$, 
  \begin{equation}\label{eq:numz2}
    \int_{0}^{\infty}\e^{\la s}\nu_{\al,z}(\dd s)<\infty.
  \end{equation} 
\end{theorem}

The proof of Theorem \ref{th:numz} is based on the local convergence of the
$2n\times 2n$ Hermitian matrix
\begin{equation}\label{eq:bip1}
  B_z := \begin{pmatrix} 0 & M -z \\ M^\top - \bar z & 0 \end{pmatrix}
\end{equation}
to the self-adjoint operator associated to a rooted random weighted infinite
tree. In particular, the measure $\nu_{\al,z}$ will be interpreted as the
expected value of the spectral measure associated to this random rooted tree.
While this line of reasoning is entirely parallel to the arguments introduced
in \cite{BCChnh}, 
an important difference here is that the resulting tree is the outcome of a
branching process where two distinct offspring distributions alternate at each
generation; see Section \ref{se:pwit} below. In contrast with \cite{BCChnh},
the row sum normalization in $M$ introduces dependencies in the random weights
of the limiting tree, making the recursive distributional equation
characterizing the spectral measure harder to analyze; see Section
\ref{se:rde} below.

Note that in contrast with the eigenvalues, the distribution of the singular
values of $M$ has unbounded support. On the other hand, unlike the case of
singular values of i.i.d.\ heavy tailed matrices \cite{BDG09,BCChnh}, it has finite
exponential moments. 

Next, we turn to the empirical distribution of the eigenvalues of $M$:
\begin{equation}\label{eq:mu}
\mu_M:= \frac 1 n \sum_{k=1} ^n \delta_{\la_k}
\,.%
\end{equation} 

\begin{theorem}[Eigenvalues]\label{th:girko}
  If $(H1)$ and $(H2)$ hold then there exists a probability measure
  $\mu_\al$ on the unit disc $\dD$,
  depending only on $\al$,  such that almost surely
  \begin{equation}\label{eq:mu1} 
    \mu_{M} \underset{n\to\infty}{\weak} \mu_\al.
  \end{equation}
  Moreover the probability measure $\mu_\al$ is isotropic, and it is neither concentrated at zero nor at the boundary of $\dD$.  
\end{theorem}

The proof of Theorem \ref{th:girko} is based on Girko's Hermitization method
\cite{MR773436}, that is we use logarithmic potentials together with the
convergence of the singular values of $M-z$ expressed by Theorem
\ref{th:numz}; see \cite{bordenave-chafai-changchun} for a survey of the
method. A crucial point of this approach is the uniform integrability of the
logarithmic function for the measures $\nu_{M,z}$, as $n\to\infty$. Two key
facts are needed to establish this property. The first step is the proof that
for almost all $z\in\dC$ the smallest singular value of $M-z$ is bounded below
by an inverse polynomial of $n$ with high probability; see Theorem \ref{th:sn}
below. This estimate uses the assumption (H2). The second step is an
adaptation of the Tao and Vu analysis \cite{MR2409368} of the singular values
in the bulk. These steps are approached by a combination of methods introduced
in \cite{BCChnh} and \cite{BCCm}. In both cases, however, the present setting
requires a nontrivial extension of the known arguments.

We refer to Figure \ref{fi:eig} for simulation plots of the spectrum of $M$.

\begin{remark}[Edge behavior: conjectures]
  The pictures in Figure \ref{fi:eig} prompt the conjecture that the spectral
  measure $\mu_\al$ is supported on a disc $\dD_\al:=\{z\in\dC:\;|z|\leq
  r_\al\}$ for some $r_\al\in(0,1)$, where $r_\al\to 1$ as $\al\to0$ and
  $r_\al\to 0$ as $\al\to 1$. A closer look at simulations actually suggests
  an even stronger conjecture, namely that with high probability, except for
  the eigenvalue $\la_1=1$, the Markov matrix $M$ has a spectral radius
  strictly less than $1$:
  \begin{equation}\label{eq:gap} 
    \max\{|\la_2|,\ldots,|\la_n|\}\leq r_\al\,,
  \end{equation}
  for some constant $r_\al\in(0,1)$ as above. Heuristic arguments seem to suggest that $r_\al\sim\sqrt {1-\al}$. A somewhat related question is
  the long time behavior of the Markov chain with transition matrix $M$. This
  question is addressed in \cite{BCS2}; see also \cite{BCS1} for related
  recent progress concerning sparse random directed graphs.
\end{remark}

\begin{figure}[htbp]
  \includegraphics[width=.7\textwidth]{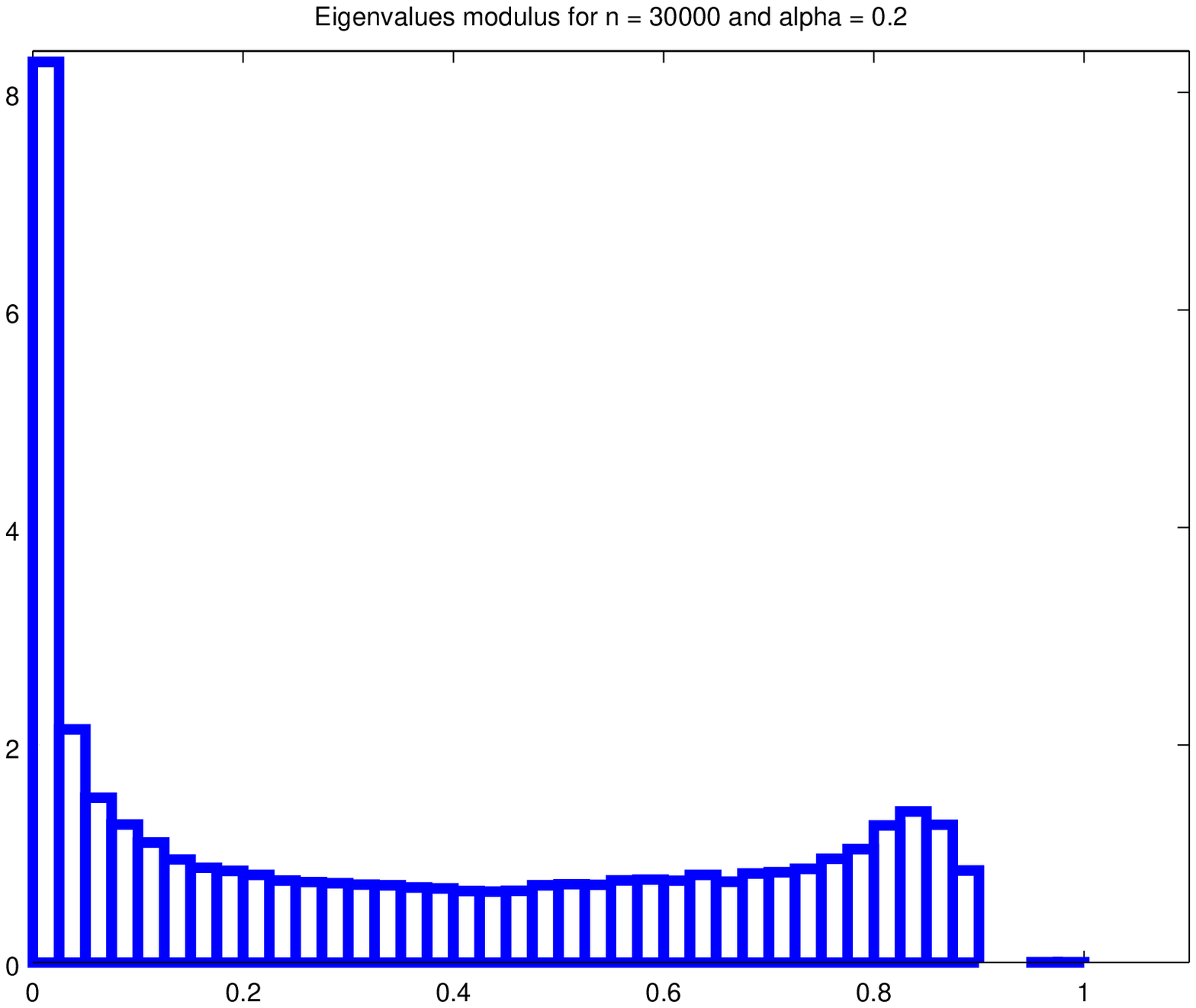}
  \includegraphics[width=.7\textwidth]{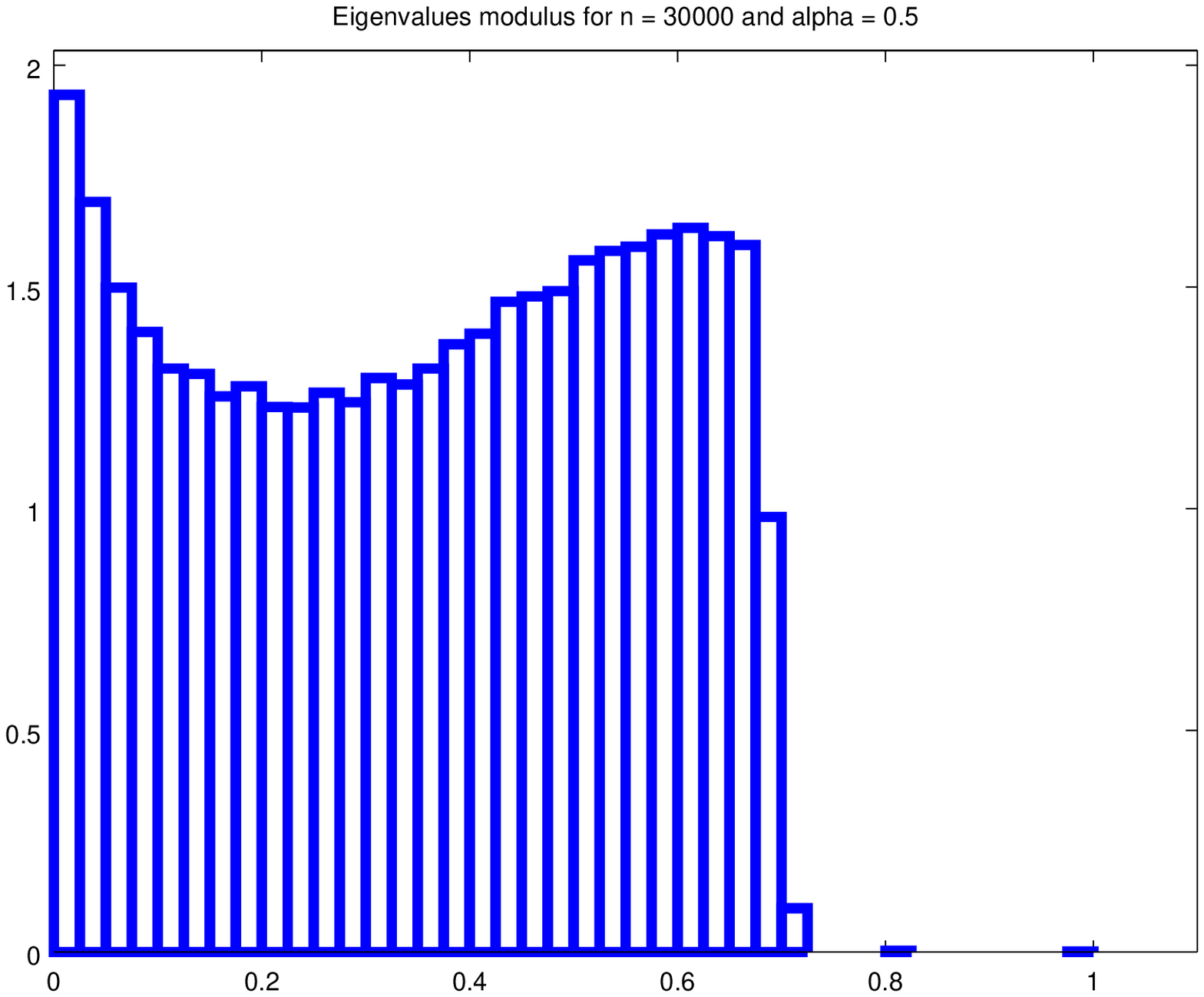}
  \includegraphics[width=.7\textwidth]{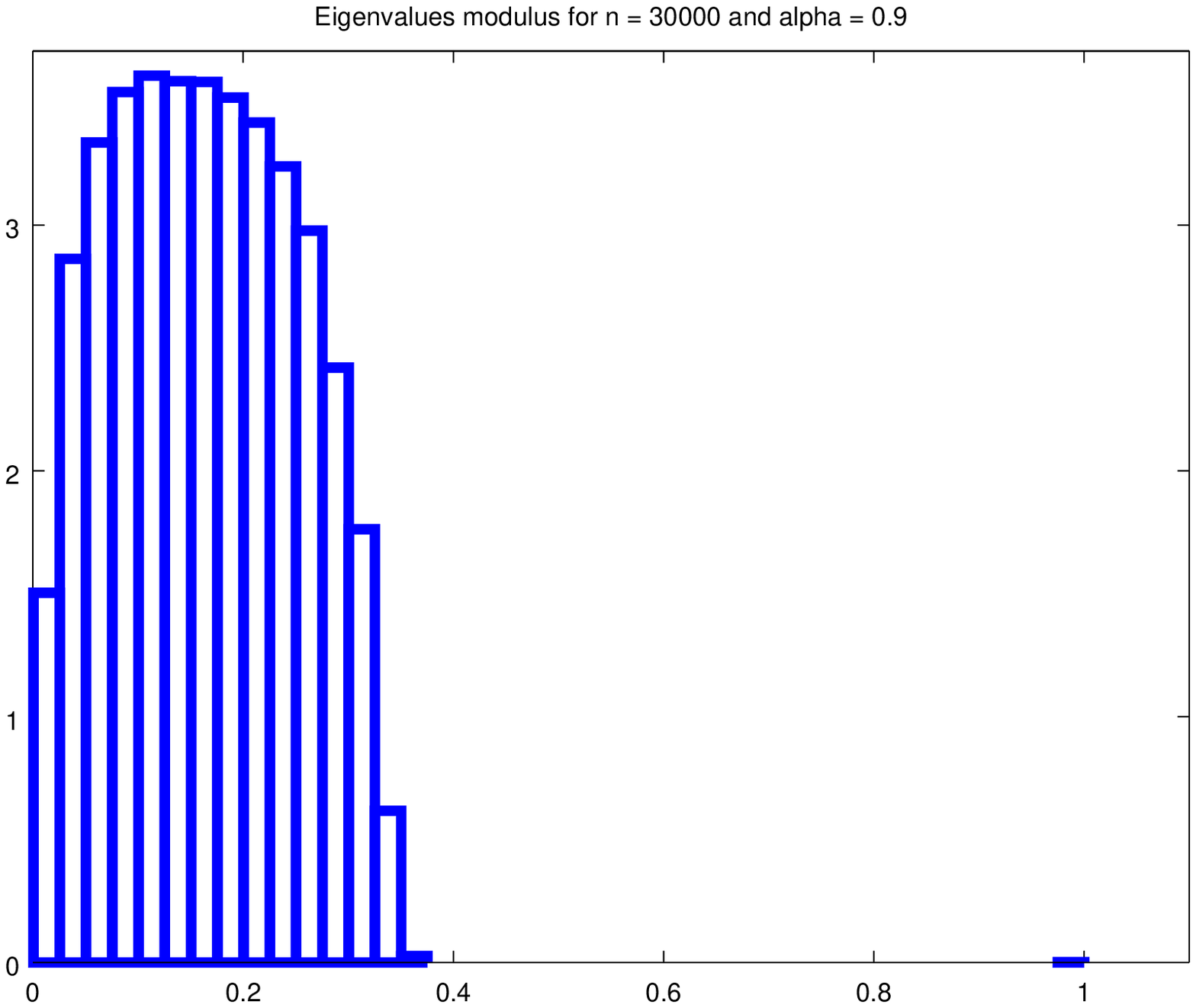}
  \caption{Histogram of the modulus of the spectrum of a single simulation of
    $M$ in dimension $n=30000$ and with tail index $\alpha$ equal to $0.2$,
    $0.5$, and $0.9$ respectively, from top to bottom. The simulation uses  $\bx= U^{-1/\al}$, where
$U$ is uniform in $[0,1]$. 
  }
  \label{fi:eig}
\end{figure}

A remark concerning the case where the variable $\bx$ satisfies (H1) with
$\al\in(1,2)$: in analogy with results from \cite{BCChr,BCChnh} we expect that
in this case the rescaled eigenvalues $ \la_k':= n^{1-1/\al}
\la_k$ satisfy
\begin{equation}\label{eq:mno}
  \frac 1 n \sum_{k=1} ^n \delta_{\la_k'}\underset{n\to\infty}{\weak} \mu_\al'
\,,%
\end{equation} 
where $ \mu_\al'$ is the isotropic probability measure with unbounded support
on $\dC$, associated to the i.i.d.\ matrix $X$; see \cite[Th.~1.2]{BCChnh}.
The case $\al=1$ should require a logarithmic correction.

The rest of this article is organized as follows. In Section \ref{se:pwit} we
obtain some preliminary properties of the model and describe the alternate
Poisson weighted tree that will be used in the proof of Theorem \ref{th:numz}. The
latter is given in Section \ref{se:conv_singular}. Finally, in Section
\ref{se:conv_spec} we prove Theorem \ref{th:girko}.

\section{Local convergence to the alternate Poisson weighted tree}
\label{se:pwit}

The key conceptual step is to associate to our matrix $M$ a limiting random
weighted tree. We follow what is by now a well established strategy; see
\cite{aldous92,aldoussteele,BorLel,BCChr}. We start with some facts on convergence to
Poisson point processes for sequences of heavy tailed random variables; see
e.g.\ \cite{Resnick} for more background.

\subsection{Convergence of rows and columns of the Markov matrix}

If $0\leq x_1\leq x_2\leq \cdots$ is distributed as the Poisson point process
with intensity $1$ on $[0,\infty)$, then
\begin{equation}\label{eq:pppa}
\xi_i=x_i^{-1/\al},
\end{equation} 
is distributed as the Poisson point process with intensity $\phi(t)=\al
t^{-\al-1}$ on $[0,\infty)$. In this case we say that the ranked vector
$\{\xi_i\}$ has law PPP($\al$). Moreover, for $\al\in(0,1)$, the variable
$S=\sum_{i=1}^\infty \xi_i$ has the one sided $\al$-stable law, with Laplace
transform
\begin{equation}\label{eq:onesided}
\bbE[\exp{(-\te S)}] = \exp{(-\G(1-\al)\te^\al)}\,,\quad \te\geq 0.
\end{equation} 
The law of the normalized ranked vector 
\begin{equation}\label{eq:pda}
  (\zeta_1,\zeta_2,\dots): = S^{-1}(\xi_1,\xi_2,\cdots),
\end{equation} 
is called the Poisson-Dirichlet law with index $\al$; see \cite{PitmanYor} for
a detailed account. We shall refer to it as PD($\al$). The next lemma
summarizes the key convergence results for the rows of the matrix $M$. We
recall that here convergence in distribution of a ranked vector coincides with
weak convergence, for every fixed $k\in\bbN$, of the joint law of the first
$k$ maxima of the vector. We write $[n]$ for the set $\{1,\dots,n\}$.

\begin{lemma}[Poisson Point Processes and Poisson Dirichlet
  distributions]\label{le:rows}
  Let $X_{i,j}$, $i,j\in[n]$ be i.i.d.\ copies of a random variable $\bx$
  satisfying (H1). Define $\rho_i=\sum_{j=1}^nX_{i,j}$ and
  $M_{i,j}=\rho_i^{-1}X_{i,j}$. Set $a_n:=(c\,n)^{1/\al}$, where $c$ is as in
  \eqref{eq:H}. Then
  \begin{enumerate}[1.]
  \item The ranked rearrangement of $\{a_n^{-1}X_{1,j},\, j\in[n]\}$ converges
    in distribution to $\{\xi_i\}$ which is a Poisson Point Process
    PPP($\al$);
  \item The random variable $a_n^{-1}\rho_1$ converges in distribution to $S$
    given in \eqref{eq:onesided};
  \item The ranked rearrangement of $\{M_{1,j}, \,j\in[n]\}$ converges in
    distribution to $\{\zeta_i\}$ which follows the Poisson-Dirichlet law
    PD($\al$).
  \end{enumerate} 
\end{lemma}

The above lemma is well known; we refer e.g.\ to \cite[Lem.~2.4]{BCChr} for a
proof. We turn to column vectors. Here the convergence result is less
immediate. We start with an auxiliary lemma. While Lemma \ref{le:rows} does
not require the full strength of the assumption (H1), the usual regular
variation assumption being sufficient \cite{BCChr}, below we do use (H1) in a
more stringent way.

\begin{lemma}
\label{le:app1} 
  In the setting of Lemma \ref{le:rows}, for all fixed $t>0$ one has
  \begin{equation}\label{eq:app11}
    \lim_{n\to\infty}n\,\bbP\left(
      X_{1,1}>\tfrac{t}{1+t}\,\rho_1\right) = \g\,t^{-\al},
  \end{equation}
  where $\g$ is the constant 
  \[
  \g=\g(\al):=\frac1{\G(1+\al)\G(1-\al)}.
  \]
\end{lemma}

\begin{proof}
  For $u> 0$, write $ \bbP(\bx>u)=L(u)u^{-\al}$, for some function $L(u)$ such
  that $L(u)\to c$, $u\to\infty$, as in \eqref{eq:H}. Denote by $\mu_n$ the
  law of the variable $a_n^{-1}\hat\rho_1$, where $\hat
  \rho_1:=\sum_{i=2}^nX_{1,i}$. Since $X_{1,1}>\tfrac{t}{1+t}\,\rho_1$ is
  equivalent to $X_{1,1}>t\hat \rho_1$, by the independence of $X_{1,1}$ and
  $\hat\rho_1$ one has
  \begin{equation}\label{eq:asyo3}
    n\,\bbP\left(
      X_{1,1}>\tfrac{t}{1+t}\,\rho_1\right)  = n\,\bbP\left(X_{1,1}>t  \hat\rho_1\right)
        =c^{-1}t^{-\al}\int_0^\infty s^{-\al}L(a_n ts) \mu_n(\dd s).
  \end{equation}
  From Lemma \ref{le:rows} we know that $\mu_n$ converges weakly to $\mu$, the
  law of $S$ in \eqref{eq:onesided}. We are going to prove that for any $t>0$:
  \begin{equation}\label{eq:trun3}
    \lim_{n\to\infty}c^{-1}\int_0^\infty s^{-\al}L(a_n ts) \mu_n(\dd s) = \int_0^\infty s^{-\al} \mu(\dd s).
  \end{equation} 
  Notice that \eqref{eq:trun3} proves \eqref{eq:app11} since by direct computation,
  using \eqref{eq:onesided} and $\al\G(\al)=\G(1+\al)$,
  \begin{align*}
    \int_0^\infty s^{-\al} \mu(\dd s)
    &= \frac1{\G(\al)}\int_0^\infty\dd x\,x^{\al-1}\int_0^\infty \e^{-xs} \mu(\dd s)\\ 
    &= \frac1{\G(\al)}\int_0^\infty\dd x\, x^{\al-1}\e^{-x^\al\G(1-\al)} = \g.
 \end{align*}

 To prove \eqref{eq:trun3}, fix $\veps>0$ and start by observing that
 \begin{equation}\label{eq:asy3}
   \lim_{n\to\infty}\int_\veps^\infty s^{-\al} L(a_n ts)\mu_n(\dd s) %
   = c \int_\veps^\infty s^{-\al} \mu(\dd s),
 \end{equation}
 Indeed, $|L(a_n ts)-c|\to 0$ as $n\to\infty$, uniformly in $s\geq \veps$. Also,
 $\int_\veps^\infty s^{-\al} \mu_n(\dd s)\to \int_\veps^\infty s^{-\al} \mu(\dd s)$ by weak
 convergence. This implies \eqref{eq:asy3}.
 
 Next, we show that for any constant $K>0$, 
  \begin{equation}\label{eq:asy4}
   \lim_{n\to\infty}\int_0^{Ka_n^{-1} }s^{-\al} L(a_n ts)\mu_n(\dd s) =0.
 \end{equation}
 The obvious bound $L(u)\leq u^\al$, $u> 0$, yields $ L(a_n t s) \leq c s^\al t^\al n$. Therefore,
 \begin{equation}\label{eq:asy5}
   \int_0^{Ka_n^{-1} }s^{-\al} L(a_n ts)\mu_n(\dd s) \leq c\,t^\al n\int_0^{Ka_n^{-1} }\mu_n(\dd s).
 \end{equation}
 Now, $\int_0^{Ka_n^{-1} }\mu_n(\dd s)=\bbP(\hat\rho_1 \leq K)$, and $\hat \rho_1\geq
 \max_{i=2,\dots, n} X_{1,i}$, so that
 \begin{equation}\label{eq:asy6}
   \int_0^{Ka_n^{-1} }\mu_n(\dd s)\leq \bbP(\bx \leq  K )^n,
 \end{equation}
 which decays exponentially in $n$ for any fixed $K>0$. This, together with
 \eqref{eq:asy5}, implies \eqref{eq:asy4}.

 Finally,  consider the integral
 \begin{equation}\label{eq:asy7}
   \int_{Ka_n^{-1} }^\veps s^{-\al} L(a_n ts)\mu_n(\dd s) .
 \end{equation}
 Since $s\geq K/a_n$, \eqref{eq:H} shows that $L(a_n ts)\leq 2c$ if $K=K(t)$
 is large enough. Therefore \eqref{eq:asy7} is bounded by
 \[
 2c \int_{Ka_n^{-1}}^\veps s^{-\al} \mu_n(\dd s)\leq 2c \int_{Ka_n^{-1}}^\veps s^{-\al} \mu_n(\dd s).
 \]
 Thus, the proof of \eqref{eq:trun3} is complete once we show
 \begin{equation}\label{eq:asym}
   \lim_{\veps\to 0}\limsup_{n\to\infty} \int_{Ka_n^{-1}}^\veps s^{-\al} \mu_n(\dd s) =0\,.
 \end{equation}

 From (H1) and \cite[Lem.~3.5]{BCChnh}, there exists $\eta>0$, $p\in(0,1)$ such
 that $\bx$ dominates stochastically the product $\eta DS$ where $D$ is a
 Bernoulli($p$) variable independent of $S$ defined in \eqref{eq:onesided}.
 Thus $a_n^{-1}\hat\rho_1$ stochastically dominates $\hat S_n:=
 a_n^{-1}(D_1S_1+\cdots D_{n-1}S_{n-1} )$ where $\{D_i\}$ and $\{S_i\}$ are
 independent, $D_i$ are i.i.d.\ copies of $D$, $S_i$ are i.i.d.\ copies of
 $S$. If $\hat\mu_n$ denotes the law of $\hat S_n$, then $\hat\mu_n =
 \sum_{k=0}^{n-1} p(k,n) \nu_k$ where $p(k,n)=\binom{n-1}{k}
 p^k(1-p)^{n-1-k}$, and $\nu_k$ is the law of $\eta a_n^{-1}(S_1+\dots+S_k)$.
 By stability $S_1+\dots+S_k$ has the same law of $k^{1/\al} S$, so that
 $\nu_k$ is the law of $\eta(k/cn)^{1/\al}S$. Let $E$ be the event that
 $D_1+\cdots+D_{n-1} \geq pn/2$. We estimate
 \begin{equation}\label{eq:asym1}
   \int_{Ka_n^{-1}}^\veps s^{-\al} \mu_n(\dd s ) %
   \leq (K/a_n)^{-\al}\bbP(E^c) + \sum_{k=pn/2}^{n-1} p(k,n)
   \int_{Ka_n^{-1}}^\veps s^{-\al} \nu_k(\dd s).
 \end{equation}
 $\bbP(E^c)$ decays to zero exponentially in $n$ from the Chernoff bound. Thus
 the first term above vanishes in the limit $n\to\infty$. We now consider the
 second term. For any $k\in[pn/2,n-1]$ one has that $\nu_k$ is the law of $\la
 S$, for a constant $\la=\eta (k/cn)^{1/\al}$ such that $\la\in[a,b]$ for some
 constants $0<a<b<\infty$. Then, uniformly in $k\in[pn/2,n-1]$,
 \[
 \int_{Ka_n^{-1}}^\veps s^{-\al} \nu_k(\dd s)\leq a^{-\al}\int_{0}^{\veps/a} x^{-\al} \mu(\dd x).
 \] 
 Since $\int_{0}^{\veps} x^{-\al} \mu(\dd x)\to 0$ as $\veps\to 0$, \eqref{eq:asym1}
 implies \eqref{eq:asym}. This ends the proof of \eqref{eq:trun3}. 
\end{proof}

We can now state the main results concerning convergence of columns of $M$.

\begin{lemma}[Ranked rearrangements]\label{le:cols} 
  In the setting of Lemma \ref{le:rows}, let $\wt X_{j,1}$, $j\in[n]$ denote
  the ranked rearrangement of $\{X_{j,1}, \,j\in[n]\}$, write $\pi =\pi^n$ for
  the permutation of $[n]$ such that $\wt X_{i,1} = X_{\pi(i),1}$ for all
  $i\in[n]$ (use e.g.\ lexicographic order to break ties if necessary), and
  define $\hat M_{j,1} := M_{\pi(j),1}$.
  \begin{enumerate}[1.]
  \item For any $k\in\bbN$, $(\hat M_{1,1} ,\dots, \hat M_{k,1})$ converges in distribution to 
    $(\hat\om_1,\dots,\hat\om_k)$, defined by
    \begin{equation}\label{eq:homegas}
      \hat\om_i:=\frac{\xi_i}{\xi_i + S_i}\,,
    \end{equation}  
    where $\{\xi_i\}$ is the sequence in \eqref{eq:pppa} with law PPP($\al$), and
    $\{S_i\}$ is an independent sequence of i.i.d.\ copies of the random
    variable $S$ in \eqref{eq:onesided}.
  \item The ranked rearrangement of $\{M_{j,1}, \,j\in[n]\}$ converges in
    distribution to the ranked rearrangement of $\{\hat\om_i\}$. The latter, in
    turn has the same law of the ranked sequence
    \begin{equation}\label{eq:omegas}
      \om_i:=\frac{\xi_i}{\xi_i + q}\,,
    \end{equation}   
    with $\{\xi_i\}$ as above, and $q$ is the constant
    $q=\g^{-1/\al}=(\G(1+\al)\G(1-\al))^{1/\al}$.
  \end{enumerate}
\end{lemma}

\begin{proof}
  Define $\hat \rho_j:= \sum_{i=2}^nX_{j,i}$,  so that 
  \begin{equation}\label{eq:truno}
    M_{j,1} = \frac{X_{j,1}}{X_{j,1}+\hat\rho_j}. 
  \end{equation} 
  It follows that
  \[
  \hat M_{j,1}= \frac{X_{\pi(j),1}}{X_{\pi(j),1}+\hat\rho_{\pi(j)}}.
  \]
  Since $\{X_{\pi(j),\ell}, \ell\neq 1, j\in[n]\}$ are independent of
  $\{X_{\pi(j),1}, j\in[n]\}$, parts 1 and 2 of Lemma \ref{le:rows} imply
  convergence in distribution of $(\hat M_{1,1} ,\dots, \hat M_{k,1})$ to
  $(\hat\om_1,\dots,\hat\om_k)$ for any fixed $k$. Note that these sequences are
  not necessarily ranked.
 
  To prove part 2,
   for any $\ell,k\in[n]$, $\ell<k$, let $E_{\ell,k}(n)$ denote the event that
  the largest $\ell$ values of $ \{M_{j,1}, \,j\in[n]\}$ are not included in
  the sequence $(\hat M_{1,1} ,\dots, \hat M_{k,1})$, and call $p_{\ell,k}(n)$
  its probability. Below, we observe that for any fixed $\ell\in\bbN$,
  \begin{equation}\label{eq:omegasi}
    \lim_{k\to\infty}\limsup_{n\to\infty}p_{\ell,k}(n)=0.
  \end{equation}
  Once \eqref{eq:omegasi} is available, then it is not difficult to see  that
  the largest $\ell$ values of $\{M_{j,1}, \,j\in[n]\}$ and the largest $\ell$
  values of $\{\hat M_{j,1}, \,j\in[n]\}$ have the same limit in distribution,
  for any fixed $\ell$. It follows from part 1 that the ranked rearrangement
  of $\{M_{j,1}, \,j\in[n]\}$ converges in distribution to the ranked
  rearrangement of $\{\hat\om_i\}$.

   To prove \eqref{eq:omegasi}, we observe that the event $E_{\ell,k}(n)$ implies
  that there exists $j_*\in[n]$ such that $M_{j_*,1}$ is larger than the
  $\ell$-th maximum of $\{M_{\pi(i),1},\, i=1,\dots,k\}$ and
  $X_{j_*,1}<X_{\pi(k),1}$. In particular, there must be $j_*\in[n]$ and
  $i_*\in\{1,\dots,\ell\}$ such that $M_{j_*,1}>M_{\pi(i_*),1}$ and
  $X_{j_*,1}<X_{\pi(k),1}$. Therefore, see \eqref{eq:truno}, one must
  have 
  \[
  \hat\rho_{j_*}< X_{j_*,1}\max_{i=1,\dots,\ell}
  \frac{\hat\rho_{\pi(i)}}{X_{\pi(i),1}}
   .\] From part 1 of the lemma it follows that $\max_{i=1,\dots,\ell}
  \frac{\hat\rho_{\pi(i)}}{X_{\pi(i),1}}$ converges in distribution to
  $\max_{i=1,\dots,\ell}\frac{S_i}{\xi_i}$, and therefore, for all fixed
  $\ell\in\bbN$:
  \[
  \lim_{t\to\infty}\limsup_{n\to\infty}\bbP\left(\max_{i=1,\dots,\ell}
    \frac{\hat\rho_{\pi(i)}}{X_{\pi(i),1}}>t\right) = 0.
  \]
   Thus, it suffices to prove that for any $t>0$:
  \begin{equation}\label{eq:omegasi1}
    \lim_{k\to\infty}\limsup_{n\to\infty}
    \bbP\left(\exists j_*\in[n]: \;\hat\rho_{j_*}< t X_{j_*,1}
      \quad\text{and}\quad 
      X_{j_*,1}<X_{\pi(k),1}\right)=0.
  \end{equation} 
  From Lemma \ref{le:rows}, for any $\veps>0$, 
  \[
   \lim_{n\to\infty}\bbP(X_{\pi(k),1} > \veps a_n)=\lim_{k\to\infty}\bbP(\xi_k > \veps)=0.
  \]
  Thus, using a union bound, to prove \eqref{eq:omegasi1} we can restrict
  ourselves to the proof that for any fixed $t>0$:
  \begin{equation}\label{eq:omegassi}
    \lim_{\veps\to 0}\limsup_{n\to\infty}
    n\,\bbP\left(\hat\rho_1 < t X_{1,1}<\veps a_n\right)=0.
  \end{equation}
  However,  as in \eqref{eq:asy3} we have
  \begin{equation}\label{eq:omegassio}
    n\,\bbP\left(\hat\rho_1 < t X_{1,1}<\veps a_n\right) %
    \leq c^{-1}t^\al\int_0^{\veps}s^{-\al}L( t^{-1}a_ns)\mu_n(\dd s).
  \end{equation}
  From the estimates \eqref{eq:asy4} and \eqref{eq:asym} one obtains
  \eqref{eq:omegassi}.
  This concludes the proof of  \eqref{eq:omegasi}.
  
  It remains to prove that the ranked values of the limiting sequence have the
  same law of $\{\om_i\}$ in \eqref{eq:omegas}. To this end, define
  $W_j:=X_{j,1}/\hat\rho_j$, so that
  \begin{equation}\label{eq:trun}
    M_{j,1} = \frac{X_{j,1}}{X_{j,1}+\hat\rho_j} = \frac{1}{1+\frac1{W_j}}. 
  \end{equation} 
  The $W_j$ are i.i.d.\ random variables. From Lemma \ref{le:app1} we know that,
  as $n\to\infty$,
  \begin{equation}\label{eq:trun2}
    n\,\bbP(W_1>t)\to (qt)^{-\al} \,,\quad t>0.
  \end{equation} 
  Then, by a well known criterion, see \cite[Th.~5.3]{Resnick}, the ranked
  rearrangement of $\{q\,W_j,\,j\in[n]\}$ converges in distribution to
  $\{\xi_i\}$ with law PPP($\al$), and from \eqref{eq:trun} it follows that
  the ranked rearrangement of $(M_{j,1},\,j=1,\dots,n)$ converges in
  distribution to $\{\om_i\}$ as claimed.
\end{proof}

\subsection{The alternate Poisson weighted tree}\label{se:alt}

Here we define the limiting random rooted tree to be associated to the matrix
$M$. We start with some standard notation. The vertex set $V$ of the tree is
deterministic, and it is defined as the countable set of all points $\bk$ of
the form $\bk=(k_1,\dots,k_m)$ for some $m\in\bbN$ and $k_i\in\bbN$, for all
$i=1,\dots,m$, together with the root vertex $\eset$. If
$\bk=(k_1,\dots,k_m)\in V$, and $j\in\bbN$, we write $(\bk, j)$ or simply $\bk
j$ for the vertex $(k_1,\dots,k_m,j)\in V$. Here $\bbN$ is the set of positive integers. The interpretation is that $\bk j$
is the $j$-th child of vertex $\bk$. Thus, if $\bk=\eset$ is the root, then
$\eset j$ or simply $j$ denotes the $j$-th vertex in the first generation.
Note that for any $\bk\in V$, $\bk\neq\eset$, there is a unique pair
$(\bv_\bk,i_\bk)$, with $\bv_\bk\in V$, and $i_\bk\in\bbN$, such that $\bk =
\bv_\bk \,i_\bk$. The vertex $\bv_\bk$ is called the parent of $\bk$. If
$\bk=(k_1,\dots,k_m)$ for some $m\in\bbN$, we write $d(\bk):=m$, and say that
$\bk$ belongs to the $m$-th generation, or equivalently that $\bk$ is at
distance $m$ from the root.

Let $\xi^{(\bk)}=\{\xi_i^{(\bk)}\}$, $\bk\in V$, denote independent copies of
the Poisson process $\{\xi_i\}$ with law PPP($\al$), as in \eqref{eq:pppa}. We can
now define a random rooted tree $\cT_0$ as the undirected weighted tree with
vertex set $V$, with root at $\eset$, with edges connecting vertices
$\bk,\bk'\in V$ if and only if $\bk'=\bk j$ or $\bk=\bk' j$ for some
$j\in\bbN$, and with edges $\{\bk,\bk j\}$ carrying the weight
$\xi^{(\bk)}_j$. We often use the notation
\begin{equation}\label{eq:adj0}
  \cT_0(\bk,\bk j) = \cT_0(\bk j,\bk ) = \xi^{(\bk)}_j, 
\end{equation}  
for the weight along the undirected edge $\{\bk,\bk j\}$.
   
Next, let us define $S_\bk=\sum_{i=1}^\infty \xi^{(\bk)}_i$. Clearly, $S_\bk$,
$\bk\in V$ are independent copies of the one sided random variable with law
\eqref{eq:onesided}. Let us modify the tree $\cT_0$ as follows. We say that
$\bk\in V$ is even (respectively odd) if $d(\bk)$ is even (respectively odd).
The root counts as an even vertex. Let $\cT_+$ denote the tree with the same
vertex set and edge set as $\cT_0$ but with weights
\begin{equation}\label{eq:adj+}
  \wh \cT_+(\bk,\bk j) %
  = \wh \cT_+(\bk j,\bk ) %
  = \hat\zeta^{(\bk)}_j\ind_{\{\bk \;\text{even}\}}+\hat\om^{(\bk)}_j\ind_{\{\bk \;\text{odd}\}}, 
\end{equation}  
where 
\begin{equation}\label{eq:zetasom}
  \hat\zeta^{(\bk)}_j:=\frac{\xi^{(\bk)}_j}{ \xi^{(\bv_\bk)}_{i_\bk}+S_{\bk }}\,,\quad 
  \hat\om^{(\bk)}_j:=\frac{\xi^{(\bk)}_j}{\xi^{(\bk)}_j + S_{\bk j}}.
\end{equation}
As before, in \eqref{eq:zetasom} the pair $(\bv_\bk,i_\bk)$ is defined by the
relation $\bk = (\bv_\bk,i_\bk)$, and we use the convention that if
$\bk=\eset$, then $\xi^{(\bv_\bk)}_{i_\bk}=0$. In particular, the weights of
the edges from the root to the first generation have the law PD($\al$) as in
\eqref{eq:pda}. We remark that with these definitions one has
\begin{equation}\label{eq:zetasom1}
\hat\zeta^{(\bk)}_j=(1-\hat \om^{(\bv_\bk)}_{i_\bk})\,\frac{\xi^{(\bk)}_j}{ S_{\bk }}\,.
\end{equation}
Also, remark that the weights $\{\hat\om^{(\bk)}_j,\;j\in\bbN\}$ are
independent of all weights of the form $\wh \cT_+(\bv,\bv i)$ for any $\bv$
such that $d(\bv)<d(\bk)$. In particular, the randomness starts afresh at
every odd generation.
 
For any $\bk\in V$ one has $\hat\zeta^{(\bk)}_j\geq \hat\zeta^{(\bk)}_{j+1} $.
On the other hand the values $\hat\om^{(\bk)}_j$, $j\in\bbN$ are not
necessarily ranked. As in Lemma \ref{le:cols} we may consider the ranked
rearrangement of the weights $\{\hat\om^{(\bk)}_j\}$. Proceeding top to
bottom from the root one can then define a new tree $\cT_+$ isomorphic to
$\hat \cT_+$ and such that all weights are ranked in the sense that
\[
\cT_+(\bk,\bk j) \geq \cT_+(\bk,\bk (j+1))\,,\qquad \bk\in V,\;j\in\bbN.
\]

In a similar fashion, we may define the weighted rooted tree
\begin{equation}\label{eq:adj-}
  \wh \cT_-(\bk,\bk j) %
  = \wh \cT_-(\bk j,\bk ) %
  = \hat\zeta^{(\bk)}_j\ind_{\{\bk \;\text{odd}\}}+\hat\om^{(\bk)}_j\ind_{\{\bk \;\text{even}\}}, 
\end{equation}  
where $ \hat\zeta^{(\bk)}_j$ and $\hat\om^{(\bk)}_j$ are defined by
\eqref{eq:zetasom}. Then, by taking ranked rearrangements one defines the tree
$\cT_-$ such that
\[
\cT_-(\bk,\bk j) \geq \cT_-(\bk,\bk (j+1))\,,\qquad \bk\in V,\;j\in\bbN.
\]
\subsection{Adjacency operators on trees}\label{se:adj}
Consider the Hilbert space $\ell^2(V)$ of square integrable sequences
$\varphi:V\mapsto\dC$, with the scalar product defined by
\[
\scalar{\varphi}{\varphi'}=\sum_{\bk\in V}\overline\varphi(\bk)\varphi'(\bk).
\]
The indicator functions $\ind_\bk$ defined by $\ind_\bk(\bk')=\de_{\bk,\bk'}$
form an orthonormal basis. The weights \eqref{eq:adj0} produce almost surely a
densely defined symmetric operator $\cT_0$ acting on $ \ell^2(V)$ by the
formula
\begin{equation}\label{eq:adje0}
  \scalar{\ind_\bk}{\cT_0\ind_{\bk'}} :=
  \begin{cases}
    \cT_0(\bk,\bk j) & \text{if}\,\; \bk'=\bk j \\
    \cT_0(\bk',\bk' j) & \text{if}\,\; \bk=\bk' j \\
    0 & \text{if}\,\; \bk\not\sim\bk'
  \end{cases}
\end{equation}  
where $\bk\not\sim\bk'$ indicates that $(\bk,\bk')$ is not an edge of the tree, that is $\bk'\neq \bk j$ and $\bk\neq \bk' j$ for all $j\in\bbN$.
In the same way, using the matrix elements \eqref{eq:adj+} and \eqref{eq:adj-} we
define the operators $\wh \cT_+$ and $\wh \cT_-$ respectively, together with
their ranked versions $\cT_+$ and $\cT_-$. All the operators defined above have dense domain $\cD$, defined as the set of finitely supported vectors, and are
symmetric in $\ell^2(V)$. With a slight abuse of notation we identify them
with their closure. It is crucial for our purposes that these operators are
actually self-adjoint; see e.g.\ \cite{reedsimon} for background on
self-adjointness of a symmetric unbounded operator. 

\begin{lemma}\label{le:esssa}
  The operators $\cT_0$, $\wh \cT_\pm$, and $\cT_\pm$ are almost surely self-adjoint.
\end{lemma}

\begin{proof}
  Self-adjointness of $\cT_0$ is shown in \cite[Prop.~A.2]{BCChr}. To
  prove the other statements we may adapt the same argument. Let us consider
  the case of $\wh \cT_+$, the others being very similar. Notice that for any
  even vertex $\bk$ one has
  \begin{equation}\label{eq:adjnorm}
    \sum_{\bv\in V}\wh \cT_+(\bk,\bv)=1.
  \end{equation}
  On the other hand, for $\bk$ odd and $\kappa >0$, define the variables
  \begin{equation}\label{eq:adjtau}
    \tau_\kappa(\bk)=\inf\Big\{t\geq 0:\;\sum_{j=t+1}^\infty\om^{(\bk)}_j\leq \kappa\Big\},
  \end{equation}
  where $\{\om^{(\bk)}_j, j\in\bbN\}$ denotes the ranked rearrangement of
  $\{\hat\om^{(\bk)}_j, j\in\bbN\}$. For any $\kappa>0$, the random variables
  $ \{\tau_\kappa(\bk),\,\bk \text{ odd}\}$ are independent and identically
  distributed. Thus, the very same proof of \cite[Prop.~A.2]{BCChr} applies,
  provided we prove that for fixed $\bk$ odd, one has
  $\bbE\tau_\kappa(\bk)<\infty$ for all $\kappa>0$ and
  $\bbE\tau_\kappa(\bk)\to 0$, as $\kappa\to\infty$. Now, by Lemma
  \ref{le:cols}, $\tau_\kappa(\bk)$ has the same law of the first $t\geq 0$
  such that $\sum_{j=t+1}^\infty\frac{\xi_j}{\xi_j+q}\leq \kappa$; see
  \eqref{eq:omegas}. In particular, it is stochastically dominated by the random
  variable
  \[
  \bar \tau_\kappa:=\inf\Big\{t\geq 0:\;\sum_{j=t+1}^\infty\xi_j\leq q\kappa\Big\},
  \]
  which satisfies $\bbE\bar \tau_\kappa<\infty$ and $\bbE\bar \tau_\kappa\to
  0$, $\kappa\to\infty$, by \cite[Lem.~A.4]{BCChr}.
 \end{proof}

\subsection{The unfolding map}\label{se:unfold}

Fix a vertex $i_0\in[n]$ and two integers $b,h\in\bbN$. Let $V_{b,h}\subset V$
denote the set of $\bk\in V$ of the form $\bk=(k_1,\dots,k_\ell)$, where
$k_i\in\{1,\dots,b\}$, and $\ell\leq h$, together with the root vertex
$\eset$. That is, $V_{b,h}$ is the finite subset obtained from $V$ by 
considering
only the first $b$ children of each node, and stopping at the $h$-th generation. As usual if
$\bk,\bk'\in V_{b,h}$ we write $\bk\sim\bk'$ iff $\bk'=\bk j$ for $\bk=\bk' j$
for some $j$. We now define a map $\phi_+: V_{b,h}\to [n]$ (respectively
$\phi_-: V_{b,h}\to [n]$) revealing the alternating ranked rearrangements of
rows and columns of the matrix $X$ starting with the row labeled $i_0$
(respectively, starting with the column labeled $i_0$). The map 
\[
\phi_+: V_{b,h}\to [n],
\]
is defined as follows. Start with $\phi_+(\eset):=i_0$. Let $X_{i_0,\pi(j)}$,
$j=1,\dots,b$, denote the first $b$ terms of the ranked rearrangement of the
row $(X_{i_0,j},j\in[n], j\neq i_0)$ (as usual we break ties using the
lexicographic ordering if necessary). Set $\phi_+(j):=\pi(j)$, $j=1,\dots,b$.
Thus far we have defined the map $\phi_+$ on the root vertex and the first
generation vertices. To uncover the second generation, set $I=\{i_0\}$, and
let $X_{\pi_1(j),\pi(1)}$, $j=1,\dots,b$, denote the first $b$ terms of the
ranked rearrangement of the column $(X_{j,\pi(1)},j\in[n], j\notin I)$, and
set $\phi_+(1j):=\pi_1(j)$, $j=1,\dots,b$. Rename $I=
\{i_0,\pi_1(1),\dots,\pi_1(b)\}$, let $X_{\pi_2(j),\pi(2)}$, $j=1,\dots,b$
denote the first $b$ terms of the ranked rearrangement of the column
$(X_{j,\pi(2)},j\in[n], j\notin I)$, and set $\phi_+(2j):=\pi_2(j)$,
$j=1,\dots,b$. Repeating this procedure for the columns labeled
$\pi(3),\dots,\pi(b)$ one completes the definition of $\phi_+$ for second
generation of $V_{b,h}$. To complete the definition of the map $\phi_+$ we
proceed recursively (redefining at each step the set $I$ so that all rows
revealed so far are excluded from the ranking) in the same way, with the rule
that if we are revealing the $m$-th generation, then we look at rankings of
the rows labeled by the vertices revealed in the $(m-1)$-th generation if $m$
is odd, and rankings of the columns labeled by the vertices revealed in the
$(m-1)$-th generation if $m$ is even. This ends the definition of $\phi_+$. To
define the map $\phi_-: V_{b,h}\to [n]$ we proceed in the exact same way, with
the role of rows and columns exchanged. Equivalently one may define $\phi_-$
as the map $\phi_+$ obtained by replacing $X$ with $X^\top$.

\begin{example}\label{ex:simple}
  {\em Let us give the explicit values of $\phi_\pm$ in a simple example.
    Suppose $n=5$ and consider the realization
    \begin{equation}\label{eq:simplx}
      X=\begin{pmatrix}0.1& 3.2& 2.1& 4& 0.2\\
        0 & 1.2 & 3.3 & 3.4 & 1.7\\
        0.4 & 10.3 & 0.1 & 2 & 3\\
        0.2 & 3.1 & 1.67 & 5 & 11\\
        8 & 4.7 & 1.2 & 1.98 & 2  
      \end{pmatrix}.
    \end{equation}
    Suppose $i_0=3$ so that $\phi_\pm(\eset)=3$. The highest value in row 3 is
    $10.3$ corresponding to column $2$, thus $\phi_+(1)=2$. The second highest
    value in row 3 is $3$ corresponding to column $5$, thus $\phi_+(2)=5$. The
    two highest values in column 2 are $4.7$ and $3.2$ corresponding to row 5
    and row 1 respectively, thus $\phi_+(11)=5$ and $\phi_+(12)=1$. Next, if
    we eliminate row 1 and row 5, the two highest values in column 5 are $11$
    and $1.7$ corresponding to row 4 and row 2 respectively, thus
    $\phi_+(21)=4$ and $\phi_+(22)=2$. On the other hand, scanning first
    columns and then rows gives the map $\phi_-$ depicted below.
    \begin{equation}\label{eq:table}
      \begin{tabular}{|c|c|c|c|c|c|c|c|} \hline
        $\bk$ &$\eset$ &1 & 2 & 11& 12 & 21& 22\\\hline
        $\phi_+(\bk)$& 3&2&5&5&1&4&2\\\hline
        $\phi_-(\bk)$ & 3& 2& 1& 4& 5& 2& 1\\\hline
      \end{tabular}\;\;
    \end{equation}
  }
\end{example}

\bigskip

We now go back to the general setting. 
To link our original matrix $X$ with the limiting tree $\cT_0$ we need a further step, namely bipartization. Let $A$ denote the $2n\times 2n$ bipartite symmetric matrix
\begin{equation}\label{eq:bipx}
  A=\begin{pmatrix}0& X\\ X^\top & 0
  \end{pmatrix}.
\end{equation} 
For any $\bk\in V_{b,h}$, set 
\begin{equation}\label{eq:bipx1}
  \psi_+(\bk) :=\begin{cases}
    \phi_+(\bk)& \bk \text{ even}\\
    \phi_+(\bk)+n & \bk \text{ odd}
  \end{cases} \quad \text{ and} \quad \psi_-(\bk) :=\begin{cases}
    \phi_-(\bk)& \bk \text{ odd}\\
    \phi_-(\bk)+n & \bk \text{ even}
  \end{cases} 
\end{equation}
The above defines two injective maps $\psi_\pm:V_{b,h}\mapsto[2n]$. A network
is a collection of vertices together with undirected weighted edges between
them. Consider the finite random network $A^{b,h}_{\pm,n}$ with vertex set
$V_{b,h}$ and with weighted edge between vertices $\bu,\bv\in V_{b,h}$ defined
by $a_n^{-1}A_{u,v}$, where $u:=\psi_\pm(\bu)$ and $v:=\psi_\pm(\bv)$. Notice
that the network $A^{b,h}_{\pm,n}$ has a fixed number of vertices but the
value of edge weights depends on $n$. See Figure \ref{fi:fignet} for an
explanatory example. At any finite $n$ the network can well have cycles, but
these tend to disappear as $n\to\infty$ as the next lemma shows. Weak
convergence of finite random networks is defined as weak convergence of the
joint law of edge weights in the natural way.
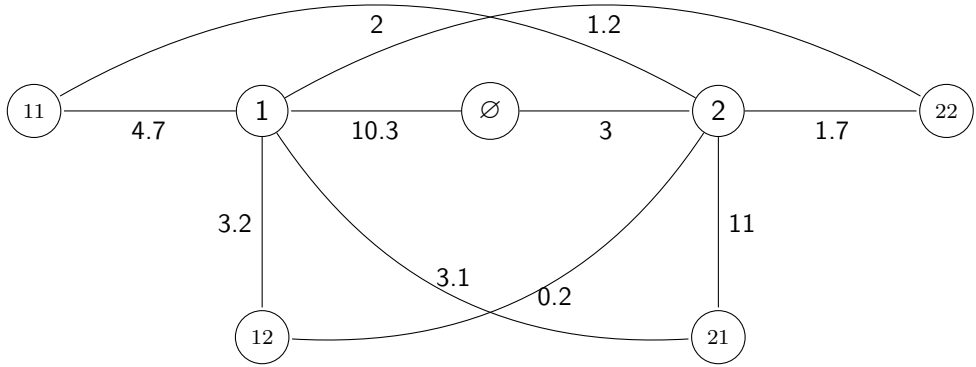
\begin{figure}
  \begin{center}
    \begin{tikzpicture}[-,>=stealth',shorten >=1pt,node distance=3cm,main node/.style={circle,draw,font=\sffamily
      }]
      \node[main node] (0) {$\eset$};
      \node[main node] (1) [left of=0] {1};
      \node[main node] (2) [right  of=0] {2};
      \node[main node] (3) [left of=1] {\SMALL{$11$}};
      \node[main node] (4) [below of=1] {\SMALL{$12$}};
      \node[main node] (5) [below of=2] {\SMALL{$21$}};
      \node[main node] (6) [right of=2] {\SMALL{$22$}};
      
      \path[every node/.style={font=\sffamily\small}]
      (0) 
      edge node [below]{3} (2)
      edge node [below]{10.3} (1)
      
      (1) 
      edge node [below]{4.7} (3)
      edge node [left]{3.2} (4)
      edge  [bend right] node [above]{3.1} (5)
      edge  [bend left]node [below]{1.2} (6)
      
      (2) 
      edge [bend right] node [below] {2} (3)
      edge [bend left]node [right]{0.2} (4)
      edge node [right]{11} (5)
      edge node [below]{1.7} (6)      
      ;
    \end{tikzpicture}
  \end{center}
  \caption{The network $A^{2,2}_{+,5}$ corresponding to the matrix $X$ in
    Example \ref{ex:simple}. If the bended edges are removed, then one obtains
    the network $\wt A^{2,2}_{+,5}$ used in the proof of Lemma
    \ref{le:localconv1}.}
  \label{fi:fignet}
\end{figure}

\begin{lemma}[Convergence of i.i.d.\ networks to trees]\label{le:localconv1}
  The random networks $A^{b,h}_{+,n}$ and $A^{b,h}_{-,n}$ both converge
  weakly, as $n\to\infty$, to the restriction of the random tree $\cT_0$ to
  the vertex set $V_{b,h}$, for any fixed $b,h\in\bbN$, and for any choice of
  the initial vertex $i_0$.
\end{lemma}

\begin{proof}
  By symmetry, $A^{b,h}_{+,n}$ and $A^{b,h}_{-,n}$ have the same law, thus it
  is sufficient to prove the statement for $A^{b,h}_{+,n}$ only. Let $\wt
  A^{b,h}_{+,n}$ denote the sub-network obtained from $A^{b,h}_{+,n}$ by giving
  weight zero to all edges of the form $\{\bu,\bv\}$ such that
  $\bu\not\sim\bv$. See Figure \ref{fi:fignet} for an example. Then, by
  construction, and using repeatedly part 1 of Lemma \ref{le:rows} one has
  that the sequence $\wt A^{b,h}_{+,n}$ converges weakly to the restriction of
  $\cT_0$ to the vertex set $V_{b,h}$. To prove the desired convergence then
  it suffices to prove that the discarded weights in $A^{b,h}_{+,n}$ converge
  to zero in probability. This can be checked by a stochastic domination
  argument as already done in \cite[Prop.~2.6]{BCChr}.
\end{proof}

Next, we turn to our matrix $M$ with normalized rows. Let $B$ denote the
$2n\times 2n$ bipartite symmetric matrix
\begin{equation}\label{eq:bipm}
  B=\begin{pmatrix}0& M\\ M^\top & 0
  \end{pmatrix}.
\end{equation} 
Let $\psi_\pm$ denote the same maps defined in \eqref{eq:bipx1}, and consider
the networks $B^{b,h}_{\pm,n}$ defined by the vertex set $V_{b,h}$ as above
but with a weighted edge between any pair of vertices $\bu,\bv\in V_{b,h}$
defined by $B_{u,v}$, where $u:=\psi_\pm(\bk)$ and $v:=\psi_\pm(\bv)$. Note
that we do not require any rescaling now. Also note that in contrast with the
case of the matrix $A$, the law of the finite networks $B^{b,h}_{+,n}$ and
$B^{b,h}_{-,n}$ do not coincide.

\begin{lemma}[Convergence of Markov networks to trees]\label{le:localconv2}
  The random network $B^{b,h}_{\pm,n}$ converges weakly, as $n\to\infty$, to
  the restriction of the random tree $\wh \cT_\pm$ to the vertex set
  $V_{b,h}$, for any fixed $b,h\in\bbN$, and for any choice of the initial
  vertex $i_0$.
\end{lemma}

\begin{proof}
  As in the proof of Lemma \ref{le:localconv1} we identify a modified
  network $\wt B^{b,h}_{\pm,n}$ which is easily seen to converge to the
  desired limit by an application of Lemma \ref{le:cols} and then prove that
  the difference between $B^{b,h}_{\pm,n}$ and $\wt B^{b,h}_{+,n}$ vanishes in
  probability. Let us start with $B^{b,h}_{+,n}$. Consider the modified
  network $\wt B^{b,h}_{+,n}$ obtained from $B^{b,h}_{+,n}$ by giving weight
  zero to any edge of the form $\{\bu,\bv\}$ such that $\bu\not\sim\bv$. Now,
  a recursive application Lemma \ref{le:cols} implies that the network $\wt
  B^{b,h}_{+,n}$ converges weakly, as $n\to\infty$, to the restriction of the
  random tree $\wh \cT_+$ to the vertex set $V_{b,h}$. There is a delicate
  point here due to the fact that while rows of $M$ are independent, the
  columns are not. The dependence stems from the normalizing sums $\rho_u$, but
  with a little care it can be shown to be negligible in the limit. We carry
  out the details below for the first two generations of the tree. Notice that
  once the convergence of $\wt B^{b,h}_{+,n}$ is established, one can conclude
  by the same stochastic domination argument mentioned in the proof of
  Lemma \ref{le:localconv1} above.
 
  Let us check the convergence of $\wt B^{b,h}_{+,n}$ in more detail. For
  simplicity of exposition we take $b=h=2$. Here the network consists of 7
  vertices, namely the $\eset, 1, 2, 11, 12, 21,22$. Let $G_n$ denote the
  network $\wt B^{2,2}_{+,n}$, and let $G_n(\bu,\bv)$ denote the weight of
  edge $\{\bu,\bv\}$. Then by definition 
  \begin{align*}
  G_n(\eset,1) &= X_{i_0,\pi(1)}/\rho_{i_0},\\ 
  G_n(\eset,2) &= X_{i_0,\pi(2)}/\rho_{i_0},\\ 
  G_n(1,11) &= X_{\pi_1(1),\pi(1)}/\rho_{\pi_1(1)},\\ 
  G_n(1,12) &= X_{\pi_1(2),\pi(1)}/\rho_{\pi_1(2)},\\ 
  G_n(2,21) &= X_{\pi_2(1),\pi(2)}/\rho_{\pi_2(1)},\\ 
  G_n(2,22) &= X_{\pi_2(2),\pi(2)}/\rho_{\pi_2(2)}.
  \end{align*}
  All other edges have weight zero. Set $\rho'_{i}:=\rho_i - X_{i,\pi(1)}$,
  $i\in[n]$, and define a new network $G'_n$ such that $G'_n$ coincides with
  $G_n$ except that on edges $\{2,21\}$ and $\{2,22\}$ one has $G'_n(2,21) =
  X_{\pi_2(1),\pi(2)}/\rho'_{\pi_2(1)}$ and $G_n(2,22) =
  X_{\pi_2(2),\pi(2)}/\rho'_{\pi_2(2)}$. Then there are no dependencies
  anymore and an application of Lemma \ref{le:cols} yields the desired
  convergence for the network $G'_n$. It remains to show that $G_n(2,21)-
  G'_n(2,21)$ and $G_n(2,22)- G'_n(2,22)$ converge to zero in probability.
  However, this follows from the fact that $a_n^{-1}X_{\pi_2(j),\pi(1)}\to 0$
  in probability for $j=1,2$.
 
\end{proof}

\subsection{Operator convergence}\label{se:opconv}

We recall a notion of local convergence that we already used in \cite{BCChr,BCChnh}. 
Let $\cD$ denote the dense subspace of finitely supported vectors in $\ell^2(V)$. 
Given a sequence of bounded self-adjoint operators $T_n$, a self-adjoint
operator $\cT$ such that $\cD$ is a core for $\cT$, and a sequence of
vertices $\bu_n\in V$, we write $(T_n,\bu_n)\to(\cT,\eset)$ if
there exists a sequence of bijections $\si_n:V\mapsto V$ such that
$\si_n(\eset)=\bu_n$ for all $n$ and such that 
\begin{equation}\label{eq:strong}
  \si_n^{-1}T_n\si_n\varphi \to \cT\varphi\,,\qquad n\to\infty,
\end{equation}
in $\ell^2(V)$, for any $\varphi\in \cD$. With slight abuse of notation, above we have used the
symbol $\si_n$ to indicate both the bijection on $V$ and the isometry
acting on $\ell^2(V)$ induced via $\si_n\ind_{\bv} = \ind_{\si_n(\bv)}$.
In Section \ref{se:spec} below, the convergence of self-adjoint operators defined above will be used to derive convergence of the corresponding spectral measures. 

Given the matrix $A$ in \eqref{eq:bipx}, we define the operator $T_n(X)$ as
follows. As usual, let $j\in\bbN$ denote the elements of $V$ belonging to the
first generation. Then, set
\begin{equation}\label{eq:biptnx}
  \scalar{\ind_{i}}{T_n(X) \ind_{j}}:=A_{i,j}\ind_{1\leq i,j\leq 2n}\,,\qquad i,j\in\bbN\,,
\end{equation}
and $\scalar{\ind_{\bu}}{T_n(X) \ind_{\bv}}=0$ for all $\bu,\bv\in V$ that are
not both in the first generation. By linearity, for every fixed $n$, \eqref{eq:biptnx} defines a random
bounded self-adjoint operator $T_n(X)$ on $\ell^2(V)$. Similarly, we define
the random bounded self-adjoint operators $T_n(M)$ by
\begin{equation}\label{eq:biptnm}
  \scalar{\ind_{i}}{T_n(M) \ind_{j}}:=B_{i,j}\ind_{1\leq i,j\leq 2n}\,,
  \qquad 
  i,j\in\bbN\,,
\end{equation}
for all $i,j\in\bbN$, and $\scalar{\ind_{\bu}}{T_n(M) \ind_{\bv}}=0$ for all
$\bu,\bv\in V$ that are not both in the first generation, where $B$ is the
matrix in \eqref{eq:bipm}.

\begin{theorem}\label{th:opconv}
  It is possible to realize the random operators $T_n(X),T^\pm_n(M)$ and
  $\cT_0,\hat\cT_\pm,\cT_\pm$ on the same probability space in such a way that
  the following holds almost surely, for all $i_0\in\bbN$, as $n\to\infty$:
  \begin{enumerate}[1)]
  \item  $(T_n(X),i_0)\to (\cT_0,\eset)$;
  \item $(T_n(M),i_0)\to (\hat \cT_+,\eset)$ and $(T_n(M),i_0+n)\to (\hat \cT_-,\eset)$.
  \item $(T_n(M),i_0)\to (\cT_+,\eset)$ and $(T_n(M),i_0+n)\to (\cT_-,\eset)$.
  \end{enumerate}
\end{theorem}

\begin{proof}
  We proceed as in \cite[Th.~2.10]{BCChnh}. By Lemma \ref{le:localconv1} and
  Skorokhod's representation theorem one can realize the operators $T_n(X)$ and
  $\cT_0$ on the same probability space in such a way that for all
  $b,h\in\bbN$, for all $\bu,\bv\in V_{b,h}$ as $n\to\infty$, almost surely
  \begin{equation}\label{eq:tnx1}
    \scalar{\ind_{\psi_+(\bu)}}{T_n(X) \ind_{\psi_+(\bv)}} %
    \to \scalar{\ind_{\bu}}{\cT_0 \ind_{\bv}},
  \end{equation}
  where the injective map $\psi_+:V_{b,h}\mapsto[2n]$ depends on $b,h$, and
  $n$. Thus, by a diagonal extraction, one can find a sequence of bijective
  maps $\si_n:V\mapsto V$ such that $\si_n(\eset)=i_0$, the fixed initial
  index, and almost surely
  \begin{equation}\label{eq:tnx2}
    \scalar{\ind_{\si_n(\bu)}}{T_n(X) \ind_{\si_n(\bv)}} \to \scalar{\ind_{\bu}}{\cT_0 \ind_{\bv}}.
  \end{equation}
  Thus, \eqref{eq:tnx2} expresses the convergence $\si_n^{-1}T_n(X)\si_n\to
  \cT_0$, in the usual sense of weak convergence of operators on a Hilbert
  space. To turn this into the required strong convergence in
  \eqref{eq:strong} it is then sufficient to prove almost sure uniform (in
  $n$) square-integrability of the vector
  $\{\scalar{\ind_\bu}{\si_n^{-1}T_n(X) \si_n\ind_\bv},\,\bu \in V\}$ for any
  fixed $\bv\in V$. The latter can be obtained as in \cite[Th.~2.10]{BCChnh}.
  This ends the proof of part 1.

  To prove part 2, we proceed in a similar way. By Lemma \ref{le:localconv2} and
  Skorokhod's representation theorem one has that for all $b,h\in\bbN$, for all
  $\bu,\bv\in V_{b,h}$ as $n\to\infty$, almost surely
  \begin{equation}\label{eq:tnx12}
    \scalar{\ind_{\psi_+(\bu)}}{T_n(M) \ind_{\psi_+(\bv)}} %
    \to \scalar{\ind_{\bu}}{\hat\cT_+ \ind_{\bv}}.
  \end{equation}
  As above one can find a sequence of bijective maps $\si_n:V\mapsto V$ such
  that $\si_n(\eset)=i_0$, and almost surely
  \begin{equation}\label{eq:tnx22}
    \scalar{\ind_{\si_n(\bu)}}{T_n(M) \ind_{\si_n(\bv)}} \to \scalar{\ind_{\bu}}{\hat\cT_+ \ind_{\bv}}.
  \end{equation}
  Now, \eqref{eq:tnx22} proves that $\si_n^{-1}T_n(M)\si_n\to\hat\cT_+$ weakly.
  The same integrability as above applies here and strong convergence follows.
  To prove that $(T_n(X),i_0+n)\to (\hat \cT_-,\eset)$, we repeat the steps
  above, recalling that now $\psi_-(\eset)=i_0+n$ and therefore the bijection
  $\si_n$ satisfies $\si_n(\eset)=i_0+n$.
  
  To prove part 3, observe that from the argument in the proof of Lemma
  \ref{le:cols} we know that there are bijections $\eta_{b',h}:V_{b',h}\mapsto
  V_{b',h}$ such that, as $b'\to\infty$, for any fixed $b,h\in\bbN$, the
  restriction of $\eta_{b',h}^{-1}\hat\cT_\pm\eta_{b',h}$ to $V_{b,h}$
  converges to the restriction of $\cT_\pm$ to $V_{b,h}$. One can then repeat
  the argument in \eqref{eq:tnx2} and \eqref{eq:tnx22} with $\psi_\pm$ replaced by
  the composed map $\psi_\pm\circ\eta_{b',h}$ with $b'=b'(n)$ a sequence with
  $b'(n)\to\infty$ as $n\to\infty$.
\end{proof}
    
\section{Convergence of singular values}\label{se:conv_singular}

Here we prove Theorem \ref{th:numz}. We shall need the following extension of
Theorem \ref{th:opconv}. For any $z\in\dC$, let $A(z)$ and $B(z)$ denote the
matrices (cf. \eqref{eq:bipx} and \eqref{eq:bipm})
\begin{equation}\label{eq:bipxz}
  A(z)=\begin{pmatrix}0& X-z\\ X^\top -\bar z & 0
  \end{pmatrix}\,,\qquad B(z)=\begin{pmatrix}0& M-z\\ M^\top -\bar z & 0
  \end{pmatrix}.
\end{equation}
Let also $T_n(X,z)$ and $T_n(M,z)$ denote the associated adjacency operators
defined as in \eqref{eq:biptnx} and \eqref{eq:biptnm} respectively with $A$ replaced
by $A(z)$ and $B$ by $B(z)$. To describe the local convergence of these new
networks we proceed as follows. Let us start with $A(z)$. We define the
limiting tree $\cT_0(z)$ recursively. Consider the previously defined tree
$\cT_0$ rooted at $\eset$. Call $\cT^1(z)$ the new tree obtained by appending
an outgoing directed edge with weight $-z$ (respectively $-\bar z$) to every
even vertex (respectively, odd vertex) of the tree, including the root. Thus
$\cT_0^1(z)$ is a tree with pending nodes (leaves) at the endpoints of the
added directed edges. Next take i.i.d.\ copies of $\cT_0$ and append one of
them to each pending node in $\cT_0^1(z)$. Now, call $\cT_0^2(z)$ the tree
obtained by appending an outgoing directed edge with weight $-z$ (respectively
$-\bar z$) to every new even vertex (respectively, odd vertex) of the tree,
where new means that it was not a vertex of $ \cT_0^1(z)$ (the leaves of
$\cT_0^1(z)$ do not count as new). Next take i.i.d.\ copies of $\cT_0$ and
append one of them to each pending node in $\cT_0^2(z)$. Recursively, we then
construct trees $\cT_0^k(z)$, $k=1,2,\dots$. Repeating ad libitum this
procedure yields the random tree that we call $\cT_0(z)$.

The trees $\hat\cT_\pm(z)$ are defined in a similar way. For $\hat\cT_+(z)$,
we start with $\hat\cT_+$ and produce the tree $\cT_+^1(z)$ by appending an
outgoing directed edge with weight $-z$ (respectively $-\bar z$) to every even
vertex (respectively, odd vertex) of the tree, including the root. Then we
append i.i.d. copies of $\hat\cT_+$ (respectively $\hat\cT_-$) to every even
(respectively odd) pending node, and we proceed recursively as above. Here a
leaf is even (respectively odd) if it is the endpoint of a directed edge
emanating from an odd (respectively even) vertex. To define $\hat\cT_-(z)$ we
proceed exactly as above with the role of even and odd interchanged. If
instead we use $\cT_\pm$, the ranked versions of $\hat\cT_\pm$, we obtain
trees that we will call $\cT_\pm(z)$.

All the above defined trees naturally define, via adjacency, linear operators
on the Hilbert space $\ell^2(V)$. To ensure the symmetry we use the convention
that when there is a directed edge $(\bk,\bk')$ with weight $z$ there is also
the opposite edge $(\bk',\bk)$ with weight $\bar z$. The following theorem is
a direct generalization of Lemma \ref{le:esssa} and Theorem \ref{th:opconv}. We
omit the details of the proof.

\begin{theorem}\label{th:opconvz}
  It is possible to realize the random operators $T_n(X,z),T^\pm_n(M,z)$ and
  $\cT_0(z)$, $\hat\cT_\pm(z)$, $\cT_\pm(z)$ on the same probability space in
  such a way that the following holds almost surely, for all fixed $z\in\dC$,
  $i_0\in\bbN$, as $n\to\infty$:
  \begin{enumerate}[1)]
  \item The operators $\cT_0(z),\hat\cT_\pm(z),\cT_\pm(z)$ are self-adjoint;
  \item  $(T_n(X,z),i_0)\to (\cT_0(z),\eset)$;
  \item $(T_n(M,z),i_0)\to (\hat \cT_+(z),\eset)$ and $(T_n(M,z),i_0+n)\to (\hat \cT_-(z),\eset)$.
  \item $(T_n(M,z),i_0)\to (\cT_+(z),\eset)$ and $(T_n(M,z),i_0+n)\to (\cT_-(z),\eset)$.
  \end{enumerate}
\end{theorem}

%
%


\subsection{The spectral measure}\label{se:spec}

For every $z\in\dC$, let $\mu_\pm^{(z)}$ denote the spectral measure at the
root of the self-adjoint operator $\hat\cT_\pm(z)$. This is the unique
probability measure on $\dR$ such that for every $\eta\in\dC_+:=\{w\in\dC:
\,\Im(w)>0\}$:
\begin{equation}\label{eq:resol}
  \scalar{\de_\eset}{(\hat\cT_\pm(z)-\eta )^{-1}\de_\eset} = m_{\mu_\pm^{(z)}}(\eta)
  = \int \frac{\mu_\pm^{(z)}(\dd x)}{x - \eta},
\end{equation}
where $(\hat\cT_\pm(z)-\eta )^{-1}$ denotes the resolvent of $\hat\cT_\pm(z)$
at $\eta\in\dC_+$, and for any probability measure $\mu$ on $\dR$ we let
$m_\mu:\dC_+\mapsto\dC_+$ denote the Cauchy-Stieltjes transform defined by
$\int\frac{\mu(\dd x)}{x - \eta}$. With these notation,
$m_{\mu_\pm^{(z)}}(\eta)$ is a bounded random variable for each
$\eta\in\dC_+$. From elementary spectral theory we know that for each fixed
$n\in\bbN$ one has
\begin{equation}\label{eq:resol1}
  \frac1{2n}\sum_{i=1}^{2n}\scalar{\de_i}{(T_n(M,z)-\eta )^{-1}\de_i} = \tr[(T_n(M,z)-\eta )^{-1}]
  = m_{\mu_n^{(M,z)}}(\eta),
\end{equation}
where $\mu_n^{(M,z)}$ denotes the empirical spectral distribution of $T_n(M,z)$, i.e.
\begin{equation}\label{eq:resol2}
  \mu_n^{(M,z)}=\frac1{2n}\sum_{i=1}^{2n}\de_{\la_{i,z}},
\end{equation}
if $\la_{i,z}\geq\la_{i+1,z} $ are the ranked eigenvalues of $B(z)$ from
\eqref{eq:bipxz}. Moreover, it is well known that $\mu_n^{(M,z)}$ is symmetric
on $\bbR$ and the singular values $s_{i,z}$ of $M-z$ satisfy
$s_{i,z}=\la_{i,z}$, $i=1,\dots,n$.

From \cite[Th.~VIII.25(a)]{reedsimon}, see also \cite[Th.~2.2]{BCChr}, it is
known that the local convergence $(T_n(M,z),i_0)\to (\hat \cT_+(z),\eset)$ and
$(T_n(M,z),i_0+n)\to (\hat \cT_-(z),\eset)$ expressed in Theorem
\ref{th:opconvz} implies convergence of resolvents, i.e.\ almost surely
\begin{gather}\label{eq:resol3}
  \scalar{\de_\eset}{(\hat\cT_+(z)-\eta )^{-1}\de_\eset} %
  = \lim_{n\to\infty}\scalar{\de_1}{(T_n(M,z)-\eta )^{-1}\de_1}\\
  \scalar{\de_\eset}{(\hat\cT_-(z)-\eta )^{-1}\de_\eset} %
  = \lim_{n\to\infty}\scalar{\de_{1+n}}{(T_n(M,z)-\eta )^{-1}\de_{1+n}}\label{eq:resol4}
\end{gather}
In particular, the above convergence holds in distributions. Since the
random variables are bounded, one has convergence of the expected values.
Thus, using linearity and exchangeability in \eqref{eq:resol1}, it follows that
for any $\eta\in\dC_+$ \begin{align} & m_{ \bbE\,\mu_n^{(M,z)}}(\eta)= \bbE
  \,m_{\mu_n^{(M,z)}}(\eta)\nonumber \\& =
  \tfrac12\bbE\scalar{\de_1}{(T_n(M,z)-\eta )^{-1}\de_1}
  + \tfrac12\bbE\scalar{\de_{1+n}}{(T_n(M,z)-\eta )^{-1}\de_{1+n}}\nonumber\\
  &\rightarrow \tfrac12\bbE\,m_{\mu_+^{(z)}}(\eta) +
  \tfrac12\bbE\,m_{\mu_-^{(z)}}(\eta) = m_{\bar \mu^{(z)}}(\eta)\,,\qquad
  n\to\infty\,,\label{eq:resol7}
\end{align}
where $\bar \mu^{(z)}$ denotes the probability measure
\begin{equation}\label{eq:resol6}
  \bar \mu^{(z)}= \tfrac12\bbE\,\mu_+^{(z)}+\tfrac12\bbE\,\mu_-^{(z)}.
\end{equation}

\subsection{Proof of Theorem \ref{th:numz}}

From \eqref{eq:resol7} we know that 
\[
m_{ \bbE\,\mu_n^{(M,z)}}(\eta)\to m_{\bar \mu^{(z)}}(\eta),
\]
for all $\eta\in\dC_+$. It is well known that convergence of the
Cauchy-Stieltjes transform implies weak convergence. Therefore
$\bbE\mu_n^{(M,z)}\weak\bar \mu^{(z)}$ as $n\to\infty$. Since $M$ has
independent rows, thanks to the general concentration result in
\cite[Lem.~C.2]{BCChnh}, it follows that $\mu_n^{(M,z)}\weak\bar \mu^{(z)}$ as
$n\to\infty$, almost surely. Finally, as observed after \eqref{eq:resol2}, $
\nu_{M,z}$ is the image of $\mu_n^{(M,z)}$ under reflection. Letting
$\nu_{\al,z}$ denote the image of $\bar \mu^{(z)}$ under reflection, we have
proved that $ \nu_{M,z}\weak \nu_{\al,z}$ as $n\to\infty$, almost surely, for
each $z\in\dC$. That
$\nu_{\al,z}$ depends on $z\in\dC$ only through
the absolute value $|z|$, and has bounded exponential moments is shown in Lemma
\ref{le:expmom} below. 
Unbounded support is proven in Lemma \ref{le:usup}. 
This will complete the proof of Theorem \ref{th:numz}. 
In Section \ref{se:rde} we shall discuss 
  the recursive
distributional equations
satisfied by the Cauchy-Stieltjes transforms $\mu_\pm^{(z)}(\eta)$. 
\subsection{Properties of the singular values distribution}


Here we use estimates on moments of $\nu_{\al,z}$ to prove that all
exponential moments of $\nu_{\al,z}$ are finite, that $\nu_{\al,z}$ depends on $z\in\dC$ only through
the absolute value $|z|$, and that the support of
$\nu_{\al,z}$ is unbounded.

\begin{lemma}[Singular values moments]\label{le:expmom}
  For any $\la>0$ and any $z\in\dC$, $\al\in(0,1)$, 
  \begin{align}\label{eq:expmom1}
    \int_0^\infty \e^{\la t}\nu_{\al,z}(dt) <\infty. 
  \end{align}
Moreover $\nu_{\al,z}=\nu_{\al,w}$ for all $z,w\in\dC$ such that  $|z|=|w|$.
\end{lemma} 
\begin{proof}  Define the moments
  \[
  m_n=\int_{0}^\infty x^n\nu_{\al,z}(\dd x).
  \] 
  %
  For $n$ even, we have
  \begin{equation}\label{eq:mns}
    m_n=\int_{-\infty}^\infty x^n\bar\mu^{(z)}(\dd x) %
    =\frac12\,\dE\scalar{\de_\eset}{\hat\cT_+(z) ^{n}\de_\eset}  %
    + \frac12\,\dE\scalar{\de_\eset}{\hat\cT_+(z) ^{n}\de_\eset}.
  \end{equation}
  Let $\G_k$ denote the set of all paths of length $n=2k$ on the tree which
  start and end at the root. Let also $w_\pm(\g)$ denote the weight of one
  path $\g\in\G_k$, so that
  \[
  \dE\scalar{\de_\eset}{\hat\cT_\pm(z) ^{2k}\de_\eset} = \sum_{\g\in\G_k} \dE [w_\pm(\g)]. 
  \]
  The weight $w_\pm(\g)$ can be written as the
  product 
  \begin{equation}\label{eq:mns0}
  \hat\cT_{\pm}(\eset,\g_1)\hat\cT_{\pm}(\g_1,\g_2)\cdots\hat\cT_{\pm}(\g_{n-1},\eset),
  \end{equation}
  where $\g_1,\dots,\g_{n-1}$ denote the ordered vertices visited by the path
  after leaving $\eset$ and before returning to $\eset$, and for simplicity we have omitted the explicit dependency on $z$ from our notation.
  
  We are going to prove that for any $k\in\dN$,
  \begin{equation}\label{eq:mnk}
    m_{2k}\leq  2^{2k} 
    C(z,k)\,,\qquad C(z,k):=\dE\left[(|z|^2+(1\vee\Psi))^k\right],
  \end{equation}
  where $\Psi:=\sum_{j=1}^\infty \om^2_j$, and $\{\om_j\}$ is defined in
  \eqref{eq:omegas}.

  Once \eqref{eq:mnk} is established we may write
  \begin{align*}
    \int_0^\infty \e^{\la t}\nu_{\al,z}(dt) 
    &\leq 2 \int_0^\infty \cosh(\la t)\nu_{\al,z}(dt)\\
    & = 2 \sum_{k=0}^\infty\frac{\la^{2k}m_{2k}}{(2k)!} %
    \leq 2 \sum_{k=0}^\infty\frac{(2\la)^{2k}C(z,k)}{(2k)!}.
  \end{align*}
  From the definition of $C(z,k)$ one has $C(z,k)\leq C(z,2k)$ and therefore
  \[
  \sum_{k=0}^\infty\frac{(2\la)^{2k}C(z,k)}{(2k)!} \leq \dE\left[\cosh(2\la(|z|^2+(1\vee\Psi) )\right]
  \leq \e^{2\la(1+|z|^2)}\dE\left[\e^{2\la \sum_k\om_k}\right],
  \]
  where we use $\cosh(t)\leq \e^{t}$, $t\geq 0$, and $(1\vee\Psi)\leq
  1+\sum_k\om_k$. On the other hand, from Campbell's formula we have, for any
  $\la\in\dR$:
  \[
  \dE\left[\e^{2\la\sum_k\om_k}\right] = \exp{\left(\int_0^\infty
      \Big(\e^{\tfrac{2\la\,t}{t+q}}-1\Big)\frac{\al \,dt}{t^{1+\al}}
    \right)}<\infty,\] which proves \eqref{eq:expmom1}.

  It remains to prove \eqref{eq:mnk}. To illustrate the computation, let us first
  consider the cases $n=2,4$ in detail. When $n=2$,
  \[
  \sum_{\g\in\G_1} \dE[ w_+(\g)] %
  = |z|^2+\sum_{k=1}^\infty \dE[\zeta_k ^2]\,, %
  \quad    \sum_{\g\in\G_1} \dE [w_-(\g)] =|z|^2+ \sum_{k=1}^\infty \dE[\om_k ^2],
  \]
  where $\{\zeta_k\},\{\om_k\}$ are defined in \eqref{eq:pda} and \eqref{eq:omegas}
  respectively. Since $0\leq \zeta_k,\om_k\leq 1$, $\sum_k\zeta_k=1$ and
  $\dE[\Psi]\leq \sum_k\dE[\om_k]=1$, it follows that 
  \[
  \sum_{\g\in\G_2} \dE[w_\pm(\g)]\leq |z|^2+\dE[1\vee\Psi]=C(z,1).
  \]
  If $n=4$ we can partition the paths in $\G_{2}$ into two families: the ones
  that reach distance $2$ from the root, say $\G_{2}^2$, and the ones that do
  not, say $\G_{2}^1$. The first term gives
  \[
  \sum_{\g\in\G_2^2} \dE [w_+(\g)]= |z|^2 \sum_{j=1}^\infty\dE\left[\om_j
    ^2\right] + |z|^2\sum_{k=1}^\infty\dE\left[\zeta_k ^2\right] +
  \sum_{k,j=1}^\infty\dE\left[\zeta_k ^2\om_j ^2\right],
  \]
  where $\{\zeta_k\},\{\om_k\}$ are independent, and
  \[
  \sum_{\g\in\G_2^2} \dE [w_-(\g)]= |z|^2 \sum_{j=1}^\infty\dE\left[\zeta_j
    ^2\right] + |z|^2\sum_{k=1}^\infty\dE\left[
    \om_k^2
  \right] +
  \sum_{k,j=1}^\infty\dE\left[\left(\tfrac{\xi_k}{\xi_k+S_k}\right)^2
    \left(\tfrac{\xi^{(k)}_j}{\xi_k+S_k}\right)^2\right],
  \]
  where $\{\xi_k\}, \{\xi^{(k)}_j\}$ are independent processes with laws
  PPP($\al$), and $S_k=\sum_{j=1}^\infty\xi^{(k)}_j$; see \eqref{eq:zetasom}. In
  particular, using $\xi^{(k)}_j/(\xi_k+S_k)\leq \xi^{(k)}_j/S_k$, one has
  \[
  \sum_{\g\in\G_2^2} \dE [w_\pm(\g)] \leq 1 +2 |z|^2\leq C(z,2).
  \]
  The second family of paths contributes
  \begin{gather*}
    \sum_{\g\in\G_2^1} \dE [w_+(\g)] =|z|^4+2|z|^2\sum_{k}^\infty\dE[\zeta_k^2]+ \sum_{k,j=1}^\infty\dE[\zeta_k^2\zeta_j^2], \\
    \sum_{\g\in\G_2^1} \dE [w_-(\g)] =|z|^4+2|z|^2\sum_{k}^\infty\dE[\om_k^2]+
    \sum_{k,j=1}^\infty\dE[\om_k^2\om_j^2].
  \end{gather*}
  Using again  $0\leq \zeta_k,\om_k\leq 1$, 
  \[ 
  \sum_{\g\in\G_2^1} \dE [w_\pm(\g)]
  \leq \dE\left[(|z|^2+(1\vee \Psi))^2\right] = C(z,2).
  \] 
  To estimate the general case $n=2k$, let $\cD_k$ denote the set of Dick
  paths of length $2k$, i.e.\ the set of all $\dZ_+$ paths staring at $0$ and
  ending at $0$ after $2k$ steps. In particular,
  $|\cD_k|=\tfrac1{k+1}\binom{2k}{k}$. 
  The set $\G_k$ can be partitioned according to the distance from the root at
  each step. This is encoded by a single element of $\cD_k$, with the natural
  correspondence: say $\G_k^\eta$, where
  $\eta=(0,\eta_1,\dots,\eta_{2k-1},0)\in\cD_k$, is the set of all paths
  $\g\in\G_k$ such that after $j$ steps, $\g$ has distance $\eta_j$ from the
  root, for all $j=1,\dots,2k-1$. Then
  \[
  \sum_{\g\in\G_k} \dE [w_\pm(\g)]  = \sum_{\eta\in\cD_k}\sum_{\g\in\G^\eta_k} \dE [w_\pm(\g)].
  \]
  For a fixed $\eta$, we can write the path $\g\in \G^\eta_k$ as
  $\g=(\eset,\g_1,\dots,\g_{2k-1},\eset)$ where $\g_i$ is the label of the
  vertex after $i$ steps, so that $\g_i$ is at distance $\eta_i$ from the
  root. Now consider the set of indexes $i_1,\dots,i_\ell$ such that
  $\eta_{i_m}=\max\{\eta_1,\dots,\eta_{2k-1}\}$. Reasoning as above we can
  write
  \begin{align*}
    \sum_{\g_{i_1}\,\dots,\g_{i_\ell}}\dE [w_\pm(\g)] 
    & \leq\dE[w_\pm(\g\smallsetminus(\g_{i_1},\dots,\g_{i_\ell}))]C(z,\ell),
  \end{align*}
  where we use the notation $\g\smallsetminus(\g_{i_1},\dots,\g_{i_\ell})$ for the path
  $\g$ once the vertices $\g_{i_1}\,\dots,\g_{i_\ell}$ have been removed.
  Notice that this is now a path of length $2k-2\ell$. Then we may proceed by
  recursion.  Using $C(z,\ell)C(z,m)\leq C(z,\ell+m)$, which follows by H\"older inequality,  we obtain
  \begin{align*}
    \sum_{\g\in\G^\eta_k} \dE [w_\pm(\g)] \leq C(z,k),
  \end{align*} 
  for every fixed $\eta\in\cD_k$. In conclusion this proves
  \begin{equation}\label{eq:reco}
    m_{2k}\leq |\cD_k|\,C(z,k)\,.
  \end{equation}
  This ends the proof of \eqref{eq:expmom1}.
  
  To show the last assertion in the lemma, notice that it suffices to show that $\bar\mu^{(z)}=\bar\mu^{(w)}$, 
  for all $z,w\in\dC$ with $|z|=|w|$. Clearly, odd moments of $\bar\mu^{(z)}$ are equal to zero. Since $\bar\mu^{(z)}$ has finite exponential moments, it is uniquely determined by the 
  numbers $m_n$, $n$ even, as in \eqref{eq:mns}; see e.g.\ \cite{reedsimon} for sufficient conditions in the classical moments problem. From the expression of the path weights $w_\pm(\g)$ in \eqref{eq:mns0}, we see that  each edge weight $u$ of the tree  appears only through the value $|u|^2$. Therefore, the moments $m_n$ do not change if we replace $z$ by $w$ with $|z|=|w|$.   
\end{proof}

\begin{lemma}[Support]\label{le:usup}
  For any $z\in\dC$, $\al\in(0,1)$, the support of $\nu_{\al,z}$ is unbounded.
\end{lemma}

\begin{proof}
  It suffices to show $m_{2k}^{1/k}\to\infty$, $k\to\infty$, where $m_n$ is
  defined in \eqref{eq:mns}. Notice that
  \[
  m_{2k}\geq \frac12 \sum_{\g\in\G_{k}}\dE[w_-(\g)]\geq \frac12\dE[\Psi^k],
  \]
  where the bound is obtained by restricting to paths that are always at
  distance at most $1$ from the root and we write $\Psi=\sum_{j=1}^\infty
  \om_{j}^2$ as before.
  Now, $\Psi$ is an unbounded random variable, and $\dP(\Psi>t)>0$ for all
  $t>0$. Indeed, a simple estimate on $\dP(\Psi>t)$ can be obtained by
  observing: for any $n\geq 1$, $\Psi\geq n\om_n^2$, and $\om_n^2\geq
  (1+q)^{-2}$ if $\xi_n=q\om_n/(1-\om_n)\geq 1$; from \eqref{eq:pppa} one has that the event $\xi_n\geq 1$ has
  probability at least $\e^{-1}/n!$ so that for all $n\in\dN$,
  \[
  \dP(\Psi\geq n(1-q)^{-2})\geq \frac1{e\,n!}.
  \] Therefore
  $m_{2k}\geq \frac12\,\dP(\Psi>t)\,t^k$, and
  $\liminf_{k\to\infty}m_{2k}^{1/k}\geq t$. Letting $t\to\infty$ proves the
  claim.
\end{proof}

\subsection{Recursive distributional equations}\label{se:rde}
Let us first recall some important formulas for resolvents on a tree. 
Let $\cA$ denote a self-adjoint operator defined by a weighted tree via adjacency; see Section \ref{se:adj}.   In particular, we write $\cA(\bu,\bv)=\overline{\cA(\bv,\bu)}$ for the complex valued weight of the directed edge $(\bu,\bv)$ if $\bv=\bu j$ or $\bu=\bv j$ for some $j\in\bbN$ and $\cA(\bu,\bv)=0$ otherwise. For $\eta\in\dC_+$, we let $$R^\eta = \scalar{\de_\eset}{(\cA-\eta)^{-1}\de_\eset}$$ denote the resolvent's diagonal entry at the root. 
Let also $u_j = \cA(\eset,j)$ denote the weight from the root to the $j$-th child. If we remove the $j$-th edge  from
the root, i.e.\ the edge with weight $u_j$, we are left with two disconnected weighted trees, one rooted at $\eset$ and one rooted at $j$. Let $\bar \cA_j$ and $\cA_j$ denote the associated self-adjoint operators respectively.
For any $\eta\in\dC_+$, we let $\bar R_j^\eta$ and $R_j^\eta$ denote the corresponding resolvent's diagonal entries at the respective roots. We refer to \cite[Th.~4.1]{BCChr} for a proof of the following lemma.
\begin{lemma}[Resolvent recursion]\label{le:recres}
For any $\eta\in\dC_+$ the resolvents entries $R^\eta,\bar R_j^\eta$ and $R_j^\eta$
satisfy the following relations. For any $j\in\bbN$ 
\begin{gather}\label{recres1}
R^\eta = \frac{\bar R_j^\eta}{1-|u_j|^2\bar R_j^\eta R_j^\eta}.
\end{gather}
Moreover,
 \begin{gather}\label{recres2}
R^\eta = -\left(\eta + \sum_{j=1}^\infty |u_j|^2 R_j^\eta
\right)^{-1}.
\end{gather}
\end{lemma}
We are interested in the law of the random variable
\begin{equation}\label{eq:resol30}
  h^{(\pm)}(\eta,z):=\scalar{\de_\eset}{(\hat\cT_\pm(z)-\eta )^{-1}\de_\eset}
\end{equation}
It is convenient to introduce also the modified tree $\bar\cT_\pm(z)$ defined
as the connected component at the root obtained from the tree $
\hat\cT_\pm(z)$ after the removal of the edge with weight $-z$ emanating from
the root. Note that this tree still depends on $z$ since all remaining children of the
root have an edge with weight $-z$ emanating from them, and so on; see the
definition of $ \hat\cT_\pm(z)$ after \eqref{eq:bipxz}. Set
\begin{equation}\label{eq:resol31}
  \bar h^{(\pm)}(\eta,z):=\scalar{\de_\eset}{(\bar\cT_\pm(z)-\eta )^{-1}\de_\eset}
\end{equation}
\begin{proposition}[Recursive distributional equations]\label{prop:recres}
For any $\eta\in\dC_+$, $z\in\dC$, the random variables $ h^{(\pm)}:=h^{(\pm)}(\eta,z)$ and $\bar h^{(\pm)}:=\bar h^{(\pm)}(\eta,z)$ satisfy the following distributional equations:
\begin{equation}\label{eq:rde030}
  h^{(-)}\stackrel{d}{=} \frac{\bar h^{(-)}}{1-  |z|^2\bar h^{(-)}\bar h^{(+)}},\qquad h^{(+)}\stackrel{d}{=} \frac{\bar h^{(+)}}{1-  |z|^2\bar h^{(-)}\bar h^{(+)}},
\end{equation}
where $\bar h^{(-)},\bar h^{(+)}$ are taken
independent. 
Moreover, if $\psi(x):= -(\eta+x)^{-1}$, then 
\begin{equation}\label{eq:rde01}
  \bar h^{(+)}\stackrel{d}{=} \psi\left(\sum_{k=1}^\infty \zeta_k^2h_k^{(-)}\right),
\end{equation} 
and $(h^{(-)},\bar h^{(-)})$ satisfy the system of distributional equations 
\begin{gather}\label{eq:rdes}
  \bar h^{(-)} \stackrel{d}{=} \psi\left(\sum_{k=1}^\infty (\tfrac{\xi_k}{\xi_k+S_k})^2
    \,\psi\left( |z|^2 \bar h^{(-)}+ \sum_{j=1}^\infty (\tfrac{\xi^{(k)}_j}{\xi_k+S_k})^2h_j^{(-)}\right)
  \right)\\
  h^{(-)}\stackrel{d}{=} \frac{\bar h^{(-)}}{1-  |z|^2\bar h^{(-)}\psi\left(
      \sum_{k=1}^\infty \zeta_k^2h_k^{(-)}
    \right)},\nonumber
\end{gather}
where $\{h_k^{(-)}\},\bar h^{(-)}$  are independent such that $\{h_k^{(-)}\}$ are i.i.d.\ copies of $h^{(-)}$, and $$\{\xi_k, k\in\bbN\}, \;\{\xi^{(k)}_j, j\in\bbN\}_{k\in\bbN}, \;\{\zeta_k, k\in\bbN\}$$
are independent processes with laws PPP($\al$), PPP($\al$), and
PD($\al$), respectively.
\end{proposition}
\begin{proof}
In the setting of Lemma \ref{le:recres} above, with $\cA$ given by $\hat\cT_\pm(z)$,  we may identify $R^\eta$ with $h^{(\pm)}(\eta,z)$ and $\bar R_j^\eta$ with $\bar h^{(\pm)}(\eta,z)$, if $j$ represents  the child associated with the weight $-z$ emanating from the root. Notice also that in this case one has that $R_j^\eta$ is an independent copy of the random variable $\bar h^{(\mp)}(\eta,z)$.
Thus from \eqref{recres1} 
one finds the distributional identities \eqref{eq:rde030}.

Next, observe that from \eqref{recres2} one has 
\begin{equation}\label{eq:rde1}
  \bar h^{(+)}(\eta,z)= - \left(\eta + \sum_{k=1}^\infty \zeta_k^2h_k^{(-)}(\eta,z)\right)^{-1}
\end{equation}
where $\{\zeta_k\}$ has law PD($\al$), while $\{h_k^{(-)}(\eta,z), k\in\bbN\}$ is a set of i.i.d.\ copies of $h^{(-)}(\eta,z)$, independent of $\{\zeta_k\}$. This proves \eqref{eq:rde01}.

Finally, again by \eqref{recres2} one has 
\begin{equation}\label{eq:rde2}
  \bar h^{(-)}(\eta,z)= - \left(\eta  + \sum_{k=1}^\infty (\tfrac{\xi_k}{\xi_k+S_k})^2
    h_k^{(+)}(\xi_k,\eta,z)\right)^{-1}
\end{equation}
where $\{\xi_k\}$ has law PPP($\al$), $S_k=\sum_{j=1}^\infty\xi^{(k)}_{j}$ where
$\{\xi^{(k)}_j, \,j\in\bbN\}_{k\in\bbN} $ are all i.i.d.\ processes with law PPP($\al$), and
$h_k^{(+)}(\xi_k,\eta,z)$ is the resolvent entry associated to the sub-tree rooted
at vertex $k$ obtained from the tree $\bar \cT_-(z)$ by deleting the root
$\eset$ and all its descendants except for $k$. We remark that in
\eqref{eq:rde2}, given $\{\xi_k\}$, the variables $\{h_k^{(+)}(\xi_k,\eta,z),
\,k\in\bbN\}$ are independent but their law depends on $\{\xi_k\}$ in a non
trivial way. We can however apply one more time the tree recursion \eqref{recres2} and obtain,
for each $k\in\bbN$:
\begin{equation}\label{eq:rde3}
  h_k^{(+)}(\xi_k,\eta,z)= - \left(\eta + |z|^2 \bar h_0^{(-)}(\eta,z) + \sum_{j=1}^\infty (\tfrac{\xi^{(k)}_j}{\xi_k+S_k})^2h_j^{(-)}(\eta,z)\right)^{-1},
\end{equation}
where $\bar h_0^{(-)}(\eta,z)$ and $h_j^{(-)}(\eta,z)$, $j\in\bbN$, are
independent copies of $\bar h^{(-)}(\eta,z)$ and $h^{(-)}(\eta,z)$
respectively.
Using \eqref{eq:rde030}--\eqref{eq:rde3}, one obtains the system of distributional
equations \eqref{eq:rdes}.
\end{proof} 

In principle, the system \eqref{eq:rdes} should determine the law of
$h^{(-)}$ and, via \eqref{eq:rde030}-\eqref{eq:rde01}, also the law of $h^{(+)}$.
However, we do not have a proof of the uniqueness of the solutions to these distributional equations. 

We remark that in the special case $z=0$, one has $\bar h^{(-)}(\eta,0)=h^{(-)}(\eta,0)$ and the above
equations reduce to the following recursive distributional equation:
\begin{gather}\label{eq:rde0}
  h^{(-)} \stackrel{d}{=} \psi\left(\sum_{k=1}^\infty
    (\tfrac{\xi_k}{\xi_k+S_k})^2 \,\psi\left( 
      \sum_{j=1}^\infty (\tfrac{\xi^{(k)}_j}{\xi_k+S_k})^2h_k^{(-)}\right)
  \right),
\end{gather}
where $\{h_k^{(-)}\}$  and 
$\{\xi_k\}, \{\xi^{(k)}_j\}$ are as above. Once we have the law of $h^{(-)}(\eta,0)$, then $h^{(+)}(\eta,0)$ is determined by the equation
\eqref{eq:rde1} which, in the case $z=0$, becomes
\begin{gather}\label{eq:rde+}
h^{(+)} \stackrel{d}{=} \psi\left(\sum_{k=1}^\infty \zeta_k^2h_k^{(-)}\right).
\end{gather}


 
\section{Convergence of eigenvalues}\label{se:conv_spec}

Here we prove Theorem \ref{th:girko}. Once the convergence of singular values
in Theorem \ref{th:numz} has been obtained, we may follow the well known
Hermitization strategy, see e.g.\ \cite{MR2409368} and
\cite{bordenave-chafai-changchun}, which allows one to prove $\mu_{M}{\weak}
\mu_\al$ as $n\to\infty$, for some probability measure $\mu_\al$ by showing
the uniform integrability of the function $\dR\ni x\mapsto |\log(x)|$ with
respect to the family of measures $\nu_{M-z}$. More precisely, we shall
establish the following lemma.

\begin{lemma}\label{le:unifint}
  For any $a>0$, $z\in\dC$, there exists $b>0$ such that for all $n\in\dN$:
  \begin{equation}\label{eq:tot}
    \bbP\left(\supp(\nu_{M-z})\not\subset[n^{-b},n^b]\right)\leq n^{-a}.
  \end{equation}
  Moreover, for any $\veps>0$, for a.a. $z\in \dC$
  \begin{equation}\label{eq:uniintinprob}
    \lim_{\de\to 0 }\limsup_{n} \,\bbP\left(\int_{K_{\de}^{c}} |\log(x)|\nu_{M-z}(\dd x)>\veps\right)=0,
  \end{equation}
  where $K_{\de}=[\de,\de^{-1}]$.
\end{lemma}

From \eqref{eq:uniintinprob} and Theorem \ref{th:numz}, as shown in
\cite[Lem.~4.3]{bordenave-chafai-changchun}, one has that $\mu_{M}{\weak}
\mu_\al$ in probability, where $\mu_\al$ is the probability measure on $\dC$
that satisfies
\begin{equation}\label{eq:umua}
  U_{\mu_\al}(z) 
  =-\int_0^\infty \log(x)\nu_{\al,z}(\dd x),
\end{equation}
for almost all $z\in\dC$. Here, for any probability measure $\mu$ on $\dC$, the function  
\[
U_{\mu}(z)=-\int_\dC\log|w-z|\mu(dw)\,,
\]
denotes the logarithmic potential. 

To improve the above convergence to the desired statement that $\mu_{M}{\weak}
\mu_\al$ a.s.\ (almost surely), we argue as follows. As in
\cite[Lem.~4.3]{bordenave-chafai-changchun}, it suffices to show that
$U_{\mu_M}(z)$ converges almost surely. In particular, since the above facts
show that $U_{\mu_M}(z)\to U_{\mu_\al}(z)$ in probability, it is now
sufficient to prove that for a.a.\ $z\in\dC$ there exists a deterministic
sequence $L_{n}$ such that, almost surely,
\[
\lim_{n\to \infty}(U_{\mu_M}(z)-L_{n})=0.
\]
By \eqref{eq:tot}, and the Borel-Cantelli lemma, there exists $b>0$ such that,
almost surely 
\[
\supp(\nu_{M-z})\subset[n^{-b},n^{b}],
\]
if $n$ is large enough. If $f_{n}(x):=\ind_{x\in[n^{-b},n^{b}]}\log(x)$, then a.s.\ 
for every $n$ large enough one has the identity 
\[
U_{\mu_M}(z)=-\frac1n\sum_{i=1}^n\log|\la_i-z| = -\int_0^\infty
\log(s)\nu_{M-z}(\dd s)= -\int_0^\infty f_n(s)\nu_{M-z}(\dd s).
\]
Since the total variation norm of $f_n$ is $O(\log n)$ the concentration
inequality of \cite[Lem.~C.2]{BCChnh} shows that
\[
\bbP\left(|U_{\mu_M}(z) -\bbE U_{\mu_M}(z)|>\veps\right)
\leq2\,\e^{-c\,\veps \,n\,(\log n)^{-2}},
\]
for some constant $c>0$. Therefore, the conclusion follows from the
Borel-Cantelli lemma letting $L_{n}=\bbE[ U_{\mu_M}(z)]$.

The above argument completes the proof of the almost sure converge $\mu_M\weak
\mu_\al$ as $n\to\infty$, where $\mu_\al$ is a probability measure on
the unit disc $\dD$ of radius $1$, that depends only on $\al\in(0,1)$.
The fact that $\mu_\al$ is supported on $\dD$ follows from the fact that for each $n$ the measure $\mu_M$ is supported on $\dD$ by the Perron-Frobenius theorem; 
see e.g.\ \cite[Chap.\ 8]{MR1084815}. 
From the
identity \eqref{eq:umua}, and the fact that $\nu_{\al,z}$ depends radially on
$z\in\dC$ one infers that $\mu_\al$ is a radially symmetric measure. Moreover,
the finiteness of $U_{\mu_\al}(0)$ implies that $\mu_\al$ is not a Dirac mass
at the origin. Similarly, one can observe that $\mu_\al$ is not concentrated
on the unit circle $|z|=1$. To this end, we will show that 
\begin{equation}\label{sumwi}
\int_\dC |z|\,\mu_\al(dz)\leq \int_0^\infty x\, \nu_{\al,0}(\dd x).
\end{equation}
From \eqref{sumwi}, the  desired conclusion follows from
$$
\int_0^\infty x\, \nu_{\al,0}(\dd x)\leq \left(\int_0^\infty x^2 \nu_{\al,0}(\dd x)\right)^\frac12=\sqrt {m_2}<1,
$$
where the strict bound $m_2<1$ is obtained by direct computations as in Lemma \ref{le:expmom}.

To prove \eqref{sumwi} we first observe that, by an inequality of Weyl, see e.g.\
\cite[Lem.~B.5]{BCChnh}, the eigenvalues of $M$ satisfy
\begin{equation}\label{sumw}
\int_\dC |\la|\,\mu_M(\dd\la) \leq \int_0^\infty x\,\nu_M(\dd x).
\end{equation}
Thus, it remains to show that a.s.\ \begin{equation}\label{sumw2}\int_0^\infty x\,\nu_M(\dd x)\to\int_0^\infty x\,\nu_{\al,0}(\dd x),\end{equation} as $n\to\infty$. Indeed, for each $K>0$, by weak convergence, this holds if we replace $x$ by the truncation $x\wedge K$. Moreover, from Schwarz' and Markov's inequalities one has the uniform deterministic bound  
\begin{align*}
\int_K^\infty x\,\nu_M(\dd x)&\leq\left(\int_0^\infty x^2\,\nu_M(\dd x)\right)^{\frac12} \left(\int_K^\infty \nu_M(\dd x)\right)^{\frac12}\\& \leq \frac1{K}\int_0^\infty x^2\,\nu_M(\dd x)\leq \frac1K\,,
\end{align*}
where the last estimate follows from  $$\int_0^\infty x^2\,\nu_M(\dd x)=\frac1n\sum_{i,j=1}^n
M_{i,j}^2 \leq \frac1n\sum_{i,j=1}^nM_{i,j} = 1.$$
Therefore, \eqref{sumw2} follows by letting $K\to\infty$.

 This ends the proof of Theorem \ref{th:girko} assuming the validity of Lemma
\ref{le:unifint}. The rest of this section is concerned with the proof of
Lemma \ref{le:unifint}.

\subsection{Extreme singular values}\label{se:smallest}

Here we prove the first part of Lemma \ref{le:unifint}. Let $s_{1,z}\geq
\cdots \geq s_{n,z}$ denote the singular values of $M-z$. We start with a simple
upper bound on $s_{1,z}$.
Notice that 
\begin{align}\label{eq:simples}
  \int_0^\infty x^2 \nu_{M-z}(\dd x)
  &= \frac1n\sum_{i,j=1}^n |M_{i,j}-z\de_{i,j}|^2\nonumber \\
  & \leq \frac2n\sum_{i,j=1}^n M_{i,j}^2 + 2 |z|^2 \leq 2
  (1+|z|^2),
\end{align} 
where we use $\sum_{j}M_{i,j}^2\leq \sum _jM_{i,j} = 1$. This implies the
deterministic bound 
\[
s_{1,z}^2\leq 2n (1+|z|^2).
\] 
Thus \eqref{eq:tot} follows from the lower bound on $s_{n,z}$ given in Theorem
\ref{th:sn} below. Notice that $M-z$ is not invertible at $z=1$, i.e.\
$s_{n,1}=0$. Because of the different scaling, in \cite[Th.~1.4]{BCCm} it was
sufficient to prove a lower bound on $s_{n,z}$ for all $z=O(n^{-1/2})$. In our
setting instead, we need to establish a similar bound for all $z\neq 1$.
\begin{theorem}\label{th:sn}
  Assume $(H2)$. Let $s_{n,z}$ denote the smallest singular value of $M-z$.
  For any $\de\in(0,1)$, and $a>0$, there exists $b=b(a,\de)>0$ such that if
  $z\in\dC$ satisfies $|z-1|\geq \de$, and $|z|\leq \de^{-1}$, then for $n\gg 1$
  one has
  \begin{equation}\label{eq:sn1}
    \dP(s_{n,z}\leq n^{-b})\leq n^{-a}.
  \end{equation}
\end{theorem}

\begin{proof}
  Fix $a>0$ and $z\in \dC$ with $|1-z|>\de$, $|z|<\de^{-1}$. We write $M=DX$,
  where $D=\text{diag}(\rho_{1}^{-1},...,\rho_{n}^{-1})$. Thus, $M-z=DY$ where
  $Y=X-zD^{-1}$, and
  \begin{equation}\label{eq:sn2}
    s_{n,z}\geq s_n(Y) \Big(\max_{i=1,\dots,n}\rho_i\Big)^{-1},
  \end{equation}
  with $s_n(Y)$ denoting the smallest singular value of $Y$. Since $\rho_i\leq
  n\max_{j} X_{i,j}$ one has, for any $\be>0$
  \[
  \dP\Big(\max_{i=1,\dots,n}\rho_i > n^\be\Big)\leq n^2\dP(\bx > n^{\be-1})\leq n^{2-\al(\be-2)},
  \] 
  for all $n$ large enough. Taking $\be$ sufficiently large ensures that
  $n^{2-\al(\be-2)}\leq n^{-a}$. From \eqref{eq:sn2} then we see that it
  suffices to prove that there exists $b>0$ such that
  \begin{equation}\label{eq:sn3}
    \dP\left(s_n(Y) \leq n^{-b}\right)\leq n^{-a}
  \end{equation}
  A repetition of the argument in \cite[Th.~1.4]{BCCm} now shows that for all
  $u>0$:
  \begin{equation}\label{eq:sn4}
    \dP\left(s_n(Y) \leq u\right)\leq 3BK_n\,u\,,
  \end{equation}
  where $B$ is the uniform bound on the probability density of the random
  variable $\bx$, which is available thanks to the assumption $(H2)$, and
  $K_n$ is the inverse of the smallest singular value of the $n\times n$
  matrix
  \[
  A_{z}:=\ind_{n}-z
  \begin{pmatrix}
    1&0&\hskip-0.2cm\cdot\cdot\hskip-0.2cm &0\\
    \vdots &\vdots &
     &\vdots\\
    1 &0&\hskip-0.2cm\cdot\cdot\hskip-0.2cm&0
  \end{pmatrix}.
  \]
  Direct calculations, see e.g.\ \cite[Lem.~C.3]{BCCm}, show that $K_n$
  satisfies
  \begin{align}\label{eq:sn5}
    K_n^2
    &=\frac{1+(n-1)|z|^{2}+|1-z|^{2}+\sqrt{(1+(n-1)|z|^{2}+|1-z|^{2})^{2}-4|1-z|^{2}}}{2 |1-z|^{2}}.
  \end{align}
  Equation \eqref{eq:sn5} can be easily estimated to obtain e.g.
  \begin{align}\label{eq:sn6}
    K_n^2
    &\leq \frac{1+(n-1)|z|^{2}+|1-z|^{2}}{|1-z|^{2}}\nonumber\\
    &\leq 1+ \frac{1+|z|^{2}(n-1)}{\de^{2}}\leq 1+\de^{-2}+ \de^{-4}n.
  \end{align}
  Using \eqref{eq:sn6}, \eqref{eq:sn3} follows by taking $u=n^{-b}$ for $b=b(a,\de)$
  large enough in \eqref{eq:sn4}.
\end{proof}

\subsection{Moderately small singular values}\label{se:sis}

Here we prove \eqref{eq:uniintinprob}, which will conclude the proof of Lemma
\ref{le:unifint}. In view of the bound \eqref{eq:simples}, and using
$|\log(x)|\leq x^{-p}$ for all sufficiently small $x>0$, if $p>0$, we see that
\eqref{eq:uniintinprob} follows if we prove
\begin{equation}\label{eq:topon}
  \limsup_{n} \dE\left(\ind_{G_n}\int_{0}^{\infty} x^{-p}\nu_{M-z}(\dd x)\right)<\infty,
\end{equation}
for some $p>0$ and some sequence of events $G_n$ such that $\dP(G_n)\to1$,
$n\to\infty$. To prove \eqref{eq:topon}, we shall follow very closely the
strategy introduced in \cite[Sec.~3]{BCChnh}; see also
\cite[Sec.~6]{bordenave-chafai-changchun}.

The following statement can be established with a straightforward adaptation
of \cite[Prop.~3.7]{BCChnh}. Set $a_n=(cn)^{1/\al}$ as in Lemma
\ref{le:rows}.

\begin{lemma}[Distance to sub-space]\label{le:distXW2}
  Assume (H1) and (H2) and take $0<\gamma <\al/4$. Let $R_1$ be the first row
  of the matrix $a_n(M - z)$. There exists a constant $C
  >0$ 
  and an event $E$ such that for any $d$-dimensional subspace $W$ of $\dC^n$
  with $ d \leq n-n ^{1 - \gamma}$, one has
  \begin{align*}  \dE \left[   \dist^{-2} ( R_1, W) \ind_E  \right]  \leq   C(n-d)^{-2/ \al}    
    \quad\text{and}\quad
    \dP ( E^c ) \leq Cn^{-(1-2\g)/\al}
    \,.
  \end{align*}
\end{lemma}

Next, we prove \eqref{eq:topon}. Fix $i\geq 2n^{1-\g}$, and let $A$ be the
matrix of the first $n-i/2$ rows of $a_{n}(M-z)$ and let $\vte_{1}\geq \cdots
\geq \vte_{n-i/2}$ denote its singular values. By the negative second moment
identity of Tao and Vu \cite[Lem.~A4]{tao-vu-cirlaw-bis},
\[
\vte_{1}^{-2}+\cdots +\vte_{n-i/2}^{-2}=\de_{1}^{-2}+\cdots\de_{n-i/2}^{-2},
\]
where $\de_{j}$ is defined as the euclidean distance $\dist(R_j,R_{-j})$, where
$R_j$ is the $j$-th row of $A$ and $R_{-j}$ is the span of all other rows of
$A$. On the other hand by Cauchy interlacing lemma, $s_{n-i,z}\geq
\vte_{n-i}/a_{n}$, and therefore $s_{n-i,z}^{-2}\leq
\frac{2a_n^2}i\sum_{j=n-i}^{n-i/2}\vte_j^{-2}$. This implies
\begin{equation}\label{eq:tak1}
  i\,  s_{n-i,z}^{-2}\leq 2a_{n}^{2}\sum_{j=1}^{n-i/2}\de_{j}^{-2}.
\end{equation}
Reasoning as in \cite[Prop.~3.3]{BCChnh} one has that the event $F$ that
$\de_{j}\geq n^{(1-2\g)/\al}$ for all $j=1,\dots, n-i/2$ has probability at
least $1-\e^{-n^\de}$ for some $\de>0$.

Taking expectation in \eqref{eq:tak1}, we get
\begin{equation}\label{eq:sumi}
  \dE\left[ i   s^{-2}_{n-i,z}  \ind_{F}\right] 
  \leq   2a^{2}_n  n  \dE\left[\de^{-2}_1\,    \ind_{F} \right]\,,
\end{equation}
Moreover, if $E$ denotes the event from Lemma \ref{le:distXW2}, then since
$R_{-1}$ has dimension $d<n-i/2\leq n-n^{1-\g}$, we see that
\begin{equation}\label{eq:sumi2}
  \dE[ \de^{-2}_1 \,    \ind_{E}] \leq C\,i^{-2/\al}.  
\end{equation}
From \eqref{eq:sumi2} it follows that
\begin{align*}
  \dE\left[\de^{-2}_1\,    \ind_{F} \right]&\leq C\,i^{-2/\al}+ \dE\left[\de^{-2}_1\,    \ind_{F\cap E^c} \right]\\
  &\leq C\,i^{-2/\al}+ n^{-2(1-2\g)/\al} \dP(E^c)\leq C\, i^{-2/\al} +
  n^{-3(1-2\g)/\al}\,,
\end{align*}
where we use the bound $\de_{1}\geq n^{(1-2\g)/\al}$ on $F$ and the bound on
$\dP(E^c)$ from Lemma \ref{le:distXW2}. If $\g< 1/6$, then $3(1- 2\g)/\al >
2/\al$ and therefore $n^{-3(1-2\g)/\al} \leq i^{-2/\al}$. From \eqref{eq:sumi}
and recalling that $a_n=(cn)^{1/\al}$ one then obtains
\begin{align}\label{eq:crux}
  \dE\left[ s^{-2}_{n-i,z}  \ind_{F}\right] 
 \leq C'  
\left(\frac{n}i\right)^{(1+2/\al)} \,,
\end{align}
for some new constant $C'>0$. 

From \eqref{eq:crux} one can prove \eqref{eq:topon} as follows. Define the event
$G_n:=F\cap\{s_{n,z}\geq n^{-b}\}$ for some $b>0$. From the above facts and
from Theorem \ref{th:sn}, we can choose $b$ so that $\dP(G_n)\to 1$. One has,
for $0<p\leq 2$
\begin{align*}
  &\dE\left(\ind_{G_n}\int_{0}^{\infty} x^{-p}\nu_{M-z}(\dd x)\right) %
  = \frac1n\sum_{i=0}^{\lfloor 2n^{1-\g}\rfloor}
  \dE[s_ {n-i,z}^{-p}\ind_{G_n}] + \frac1n\sum_{i=\lfloor 2n^{1-\g}\rfloor +1}^{n-1}
  \dE[s_ {n-i,z}^{-p}\ind_{G_n}]\\&
  \qquad \leq  2n^{bp}n^{-\g} +  \frac1n\sum_{i=\lfloor 2n^{1-\g}\rfloor +1}^{n-1}
  \dE[s_ {n-i,z}^{-2}\ind_{F}]^{p/2} \leq 2n^{bp}n^{-\g}+ 
  \frac{C}n\sum_{i=1}^{n}\left(\frac{n}i\right)^{\frac{p}2(1+2/\al)},
\end{align*}
for some new constant $C>0$. If $\g\in(0,1/6)$ is fixed, and $b>0$ is given, then we choose $p$ such that $p<\g/b$ and $p<2\al(2+\al)$  so that  the above expression is uniformly bounded. 

\bibliographystyle{plain}
\bibliography{heavymarkov}

\end{document}